\documentclass[a4paper,reqno,11pt]{amsart}
\usepackage{amssymb}
\usepackage{amstext}
\usepackage{amsmath}
\usepackage{amscd}
\usepackage{amsthm}
\usepackage{amsfonts}
\usepackage{enumerate}
\usepackage{graphicx}
\usepackage{color}
\usepackage{here} 
\usepackage{bm} 
\usepackage{mathrsfs}
\usepackage{mathtools}
\usepackage{latexsym}
\usepackage{booktabs}
\usepackage{fancyhdr}
\usepackage{subcaption}
\usepackage{tikz}
\usepackage{tikz-cd}
\usetikzlibrary{arrows}
\usetikzlibrary{decorations.markings}
\usetikzlibrary{patterns,decorations.pathreplacing}
\usetikzlibrary{calc}
\tikzstyle{vertex}=[circle ,draw, inner sep=0pt, minimum size=6pt]

\makeatletter
\pgfdeclarepatternformonly[\LineSpace]{my north west lines}{\pgfqpoint{-1pt}{-1pt}}{\pgfqpoint{\LineSpace}{\LineSpace}}{\pgfqpoint{\LineSpace}{\LineSpace}}
{
    \pgfsetcolor{\tikz@pattern@color} 
    \pgfsetlinewidth{0.6pt}
    \pgfpathmoveto{\pgfqpoint{0pt}{\LineSpace}}
    \pgfpathlineto{\pgfqpoint{\LineSpace + 0.1pt}{- 0.1pt}}
    \pgfusepath{stroke}
}
\makeatother 
\newdimen\LineSpace
\tikzset{
    line space/.code={\LineSpace=#1},
    line space=8.5pt
}

\usepackage{hyperref}
\hypersetup{
  colorlinks=true,
  citecolor=red,
  linkcolor=blue
}

\makeatletter
  
\@addtoreset{equation}{section}
\@namedef{subjclassname@2020}{\textup{2020} Mathematics Subject Classification}
\makeatother

\newcommand{\calB}{\mathcal{B}}
\newcommand{\calC}{\mathcal{C}}
\newcommand{\calD}{\mathcal{D}}
\newcommand{\calF}{\mathcal{F}}
\newcommand{\calG}{\mathcal{G}}
\newcommand{\calH}{\mathcal{H}}

\newcommand{\calN}{\mathcal{N}}

\newcommand{\calT}{\mathcal{T}}
\newcommand{\calW}{\mathcal{W}}

\newcommand{\calZ}{\mathcal{Z}}

\newcommand{\ZZ}{\mathbb{Z}}
\newcommand{\QQ}{\mathbb{Q}}
\newcommand{\RR}{\mathbb{R}}

\newcommand{\TT}{\mathbb{T}}
\newcommand{\kk}{\Bbbk}

\newcommand{\scrA}{\mathscr{A}}
\newcommand{\scrC}{\mathscr{C}}

\newcommand{\scrO}{\mathscr{O}}

\newcommand{\bfa}{\mathbf{a}}
\newcommand{\ab}{\mathbf{a}}
\newcommand{\bb}{\mathbf{b}}

\newcommand{\mb}{\mathbf{m}}

\newcommand{\tb}{\mathbf{t}}

\newcommand{\vb}{\mathbf{v}}
\newcommand{\wb}{\mathbf{w}}
\newcommand{\bfx}{\mathbf{x}}
\newcommand{\xb}{\mathbf{x}}
\newcommand{\yb}{\mathbf{y}}
\newcommand{\zb}{\mathbf{z}}

\newcommand{\fkB}{\mathfrak{B}}
\newcommand{\fkC}{\mathfrak{C}}

\newcommand{\fkI}{\mathfrak{I}}

\newcommand{\Hom}{\operatorname{Hom}}

\newcommand{\Gro}{Gr\"{o}bner }

\newcommand{\set}[1]{\left\{ #1 \right\}}
\newcommand{\rbra}[1]{\left( #1 \right)}
\newcommand{\ceil}[1]{\left \lceil #1 \right \rceil}

\def\opn#1#2{\def#1{\operatorname{#2}}} 
\opn\Cl{Cl} \opn\conv{conv} \opn\deg{deg} \opn\rank{rank} \opn\rk{rk} \opn\Spec{Spec} \opn\Stab{Stab} \opn\aff{aff} \opn\div{div} \opn\GL{GL}
\opn\cone{cone} \opn\End{End} \opn\Hom{Hom} \opn\mod{mod} \opn\gldim{gldim} \opn\pdim{pdim} \opn\diag{diag} \opn\vert{vert}
\opn\Block{Block} \opn\Pyr{Pyr} \opn\max{max} \opn\min{min} 
\opn\ini{in} \opn\out{out}
\opn\rev{rev} \opn\ker{Ker} \opn\im{Im} \opn\lat{lat} \opn\pull{pull} \opn\rev{rev} \opn\Sym{Sym} \opn\supp{supp} \opn\inte{int} \opn\star{star} \opn\sign{sign} \opn\sortl{sortl} \opn\sortr{sortr} \opn\Sortl{Sortl} \opn\Sortr{Sortr} \opn\PLC{PLC} \opn\ehr{ehr} \opn\pt{pt}

\newtheorem{thm}{Theorem}[section]
\newtheorem{cor}[thm]{Corollary}

\newtheorem{lem}[thm]{Lemma}
\newtheorem{prop}[thm]{Proposition}

\theoremstyle{definition}
\newtheorem{definition}[thm]{Definition}

\newtheorem{ex}[thm]{Example}

\theoremstyle{remark}

\newtheorem{rem}[thm]{Remark}

\begin{document}

\title{Toric rings associated with root systems and conic divisorial ideals via matroid theory}
\author{Koji Matsushita}
\address{Graduate School of Mathematical Sciences, The University of Tokyo, Komaba, Meguro-ku, Tokyo 153-8914, Japan}
\email{koji-matsushita@g.ecc.u-tokyo.ac.jp}


\maketitle

\begin{abstract}
We study conic divisorial ideals from the viewpoint of matroid theory and apply the resulting framework to toric rings arising from root systems. 
For a toric ring, we describe the polytope representing divisor classes corresponding to conic divisorial ideals in terms of matroids.
We then turn to the toric ring $R_P$ associated with a certain subset $P$ of a classical root system, called a signed poset.
We compute the divisor class group and characterize the ($\mathbb{Q}$-)Gorenstein property of $R_P$ in terms of $P$.
Moreover, we also construct a polytope characterizing the conic divisorial ideals of $R_P$.
This recovers and extends previous results on Hibi rings to our toric rings.
\end{abstract}

\maketitle

\section{Introduction}
Throughout this paper, let $\kk$ be a field.
Let $C\subset \RR^d$ be a $d$-dimensional rational polyhedral cone.
We define the toric ring
\[
R:=\kk[C\cap\ZZ^d]
 =\kk[\tb^\mb:=t_1^{m_1}\cdots t_d^{m_d}\; : \; \mb:=(m_1,\ldots,m_d)\in C\cap\ZZ^d].
\]

Toric rings are objects of interest in various areas such as commutative algebra, combinatorics, algebraic geometry, and representation theory, and have been studied from many different viewpoints (see, for example, \cite{bruns2009polytopes,cox2024toric,herzog2018binomial}).
They enjoy several favorable ring-theoretic properties, and many of their invariants can be described explicitly.
For instance, toric rings are normal affine semigroup rings and hence Cohen--Macaulay by Hochster's theorem \cite{Hochster1972rings}.
Moreover, their canonical modules and divisor class groups admit concrete descriptions in terms of the underlying cones (see Section~\ref{sec:pre}).

For a positive integer $k$, let
\[
R^{1/k}:=\kk\bigl[C\cap \tfrac{1}{k}\ZZ^d\bigr],
\]
regarded as an $R$-module.
Then $R^{1/k}$ has a natural decomposition as an $R$-module:
\[
R^{1/k}
 =\bigoplus_{L\in \frac{1}{k}\ZZ^d/\ZZ^d} \kk\{\tb^\mb : \mb \in C\cap L\}.
\]
The modules appearing in this decomposition are called the conic modules, or conic divisorial ideals, of $R$.
Conic modules form a distinguished class of divisorial ideals, that is, rank 1 reflexive modules, on toric rings, and their applications have been extensively studied (see, e.g., \cite{bruns2005conic,bruns2003divisorial}).
It is known that the number of isomorphism classes of conic modules of $R$ is finite.
Moreover, since $R^{1/k}$ is a maximal Cohen--Macaulay (MCM, for short) $R$-module, every conic module is also an MCM $R$-module.

Conic divisorial ideals play important roles in non-commutative algebraic geometry and commutative algebra with positive characteristic.
For example, non-commutative (crepant) resolutions (NC(C)Rs, for short) of toric rings can be constructed by considering endomorphism rings of direct sums of conic divisorial ideals (see, e.g., \cite{broomhead2012dimer,faber2019non,higashitani2022conic,higashitani2019conic,ishii2016dimer,nakajima2019non,vspenko2017non}).
Moreover, conic divisorial ideals are used to analyze the structure of the Frobenius push-forward of $R$.
In particular, invariants such as the (generalized, dual) $F$-signatures and the Hilbert--Kunz multiplicity can be computed using information on conic divisorial ideals (see, e.g., \cite{bruns2005conic,higashitani2021generalized,matsushita2024dual}).

From the viewpoint of divisor class groups, conic modules admit a convenient description.
Indeed,  every toric ring $R$ can be realized as the degree ${\bf 0}$ component $S_{\bf 0}$ of a $\Cl(R)$-graded polynomial ring $S:=\kk[x_1,\ldots,x_n]$, where $\Cl(R)$ denotes the divisor class group of $R$, and that every conic module is isomorphic to $S_\ab$ for some $\ab\in \Cl(R)$.
Let $\beta_i:=\deg(x_i)$ for $i=1,\ldots,n$, and call it a weight of $R$.
Then the conic modules of $R$ are characterized by a polytope in $\Cl(R)\otimes_\ZZ \RR$ determined by the weights of $R$ (see Section~\ref{sec:pre}).

\medskip

In light of the above observations and current situation, it is important to determine the weights of $R$ and the polytope that characterizes its conic modules.
For several classes of toric rings arising from combinatorial objects, both the weights and the corresponding polytope have been explicitly determined:
\begin{itemize}
    \item Hibi rings (\cite{higashitani2019conic,matsushita2022conic}),
    \item edge rings of complete multipartite graphs (\cite{higashitani2022conic}),
    \item stable set rings of perfect graphs (\cite{matsushita2022conic}).
\end{itemize}
On the other hand, for general toric rings, the problem of determining these data is quite difficult, and, outside the classes listed above, it has not been studied in much detail.

In a previous paper \cite{matsushita2022conic}, the author suggested that the theory of oriented matroids provides an effective framework for studying the weights and the polytope characterizing conic modules of general toric rings. 
In this paper, we develop this idea in a more precise and systematic way.
As an application, we determine these data for a class of toric rings arising from root systems.

\medskip 

Let $\Phi\subset \RR^d$ be a root system.
A $\Phi$-poset $P$ is a subset of $\Phi$ satisfying certain conditions.
When $\Phi$ is a classical root system, we will also refer to such a $\Phi$-poset simply as a signed poset.
We then define the toric ring of $P$, denoted by $R_P$, to be the toric ring associated with the dual cone of the cone generated by $P$ (see Section~\ref{subsec:signed poset} for the precise definitions).

In \cite{alexander2014root}, this toric ring $R_P$ of a signed poset $P$ is called the \textit{weight semigroup ring} of $P$.
In this paper, however, we will simply refer to it as the toric ring of $P$.
Some of its properties, such as defining ideals, \Gro bases, and complete intersectionness, have already been studied (see \cite{adrien2012linear,alexander2014root,valentin2012ppartition}).
The toric ring $R_P$ was originally studied mainly from a combinatorial point of view.
From the viewpoint of commutative algebra, however, its divisor-theoretic properties seem to have received much less attention.

The toric ring $R_P$ may be regarded as a generalization of Hibi rings, which are toric rings arising from posets introduced by Hibi (\cite{hibi1987distributiv}).
For Hibi rings, various properties have been investigated, including divisor class groups (\cite{hashimoto1992divisor}), Gorensteinness (\cite{hibi1987distributiv}), and conic divisorial ideals (\cite{higashitani2019conic,matsushita2022conic,robinson2020big}).

In this paper, we extend these results on Hibi rings to toric rings of signed posets.

\medskip

We summarize our main results as follows:
\begin{itemize}
    \item We describe the facet-defining inequalities of the polytope characterizing conic divisorial ideals of a toric ring in terms of matroid theory (Theorem~\ref{thm:conicdesc}).
    Moreover, we provide a formula for the number of isomorphism classes of conic divisorial ideals of toric rings (Corollary~\ref{cor:numberconic}).
    \item We compute the divisor class groups and the weights of toric rings associated with $\Phi$-posets arising from classical root systems (Proposition~\ref{prop:CL(R_P)} and Theorem~\ref{thm:weightR_P}). Using this description of the weights, we characterize when these toric rings are ($\QQ$-)Gorenstein (Theorem~\ref{thm:Gorenstein}). In addition, we apply Theorem~\ref{thm:conicdesc} to this class of toric rings (Theorem~\ref{thm:weightconic}).
    \item We further specialize the above results to balanced signed posets.
    In particular, in this case we show that the isomorphism classes of conic divisorial ideals are in bijection with suitable acyclic orientations of the graph associated with the signed poset (Theorems~\ref{thm:weightAconic} and \ref{thm:characonicA}).
    These results may be regarded as generalizations of known facts for Hibi rings.
\end{itemize}

\medskip

This paper is organized as follows.
In Section~\ref{sec:pre}, we recall the definitions and notation concerning toric rings and their conic divisorial ideals.
In Section~\ref{sec:conictoric}, we review the necessary notation from matroid theory and state our main results for general toric rings. We study the polytope characterizing conic divisorial ideals and the number of their isomorphism classes.
In Section~\ref{sec:4}, we study toric rings arising from $\Phi$-posets of classical root systems. After recalling the necessary definitions and notation for these $\Phi$-posets and their signed graph descriptions, we compute their divisor class groups and weights. We also characterize when these rings are ($\QQ$-)Gorenstein and describe their conic divisorial ideals.
In Section~\ref{sec:balanced}, we focus on balanced signed posets. We explain the notion of balanced signed posets and establish a bijection between the isomorphism classes of conic divisorial ideals and certain acyclic orientations of an associated graph.

\subsection*{Notation and Conventions}
For a set $S$ and its subset $S \subseteq X$, let $\chi_S \in \RR^X$ denote its characteristic vector. In particular, for $s \in X$, we write $\chi_s$ for the characteristic vector of $\{s\}$, that is, the corresponding unit vector.
We naturally regard $\RR^d$ and $\ZZ^d$ as $\RR^{[d]}$ and $\ZZ^{[d]}$, respectively, where $[d]:=\{1,\ldots,d\}$. 
Unless an explicit representation is given, all vectors are understood to be column vectors.
For a finite subset $V$ of $\ZZ^d$ with $\#V=n$, 
we denote by $A_V$ the $d\times n$-matrix whose columns are the vectors in $V$.
We regard $A_V$ as the $\ZZ$-linear map from $\ZZ^n$ to $\ZZ^d$ given by matrix multiplication.
For a vector $\ab \in \RR^d$, we write $|\ab|$ for the sum of the components of $\ab$, that is,
$|\ab| := \sum_{i=1}^d \ab(i)$.
Moreover, we define the vectors $\ab^+$ and $\ab^-$ of $\RR^d$ as follows: 
\[
\ab^+(i):=\begin{cases}
    \ab(i) &\text{ if $\ab(i)\ge 0$}, \\
    0 &\text{ if $\ab(i)< 0$},
\end{cases} \qquad
\ab^-(i):=\begin{cases}
    0 &\text{ if $\ab(i)\ge 0$}, \\
    -\ab(i) &\text{ if $\ab(i)< 0$}.
\end{cases}
\]

\bigskip

\section{Preliminaries on toric rings}\label{sec:pre}

First, we recall toric rings and their conic divisorial ideals.

Let $V:=\{\vb_1,\ldots,\vb_n\}$ be a finite subset of $\ZZ^d$.
We consider the rational polyhedral cone 
$$
C(V):=\RR_{\ge 0}V=\set{\sum_{i=1}^na_i \vb_i : a_1,\ldots,a_n\in \RR_{\ge 0}} 
$$
generated by $\vb_1, \cdots, \vb_n$. 
We assume that this system of generators is minimal, the generators are primitive, i.e., $\epsilon \vb_i \notin \ZZ^d$ for any $0<\epsilon <1$, and $C(V)$ is strongly convex.
We also consider the dual cone $C^\vee(V)$ of $C(V)$: 
$$
C^\vee(V):=\{{\bf x}\in\RR^d : A^T_V\xb(i)\ge0 \text{ for all } i\in [n] \}. 
$$
We now define the \textit{toric ring} 
\begin{align*}
R=\kk[C^\vee(V)\cap\ZZ^d]=\kk[t_1^{m_1}\cdots t_d^{m_d} : (m_1, \ldots, m_d)\in C^\vee(V)\cap\ZZ^d]. 
\end{align*}
Note that $R$ is a $d$-dimensional Cohen-Macaulay normal domain. 
In addition, for each $\bfa\in\RR^n$, we set 
$$
T(\bfa)=\{{\bf x}\in\ZZ^d : A^T_V\xb(i)\ge \ab(i) \text{ for all } i\in [n] \}. 
$$
Then, we define the module $\calT(\bfa)$ generated by all monomials whose exponent vector is in $T(\bfa)$. 
By the definition, we have $T(0)=C^\vee(V)\cap\ZZ^d$ and $\calT(0)=R$.
Moreover, we note some facts associated with the module $\calT(\bfa)$ (see, e.g., \cite[Section 4.F]{bruns2009polytopes}): 
\begin{itemize}
\item Since $A^T_V\xb(i)\in \ZZ$ for any $i\in [n]$ and any $\bfx \in \ZZ^d$, we can see that $\calT(\bfa)=\calT(\lceil \bfa\rceil)$, where $\lceil \; \rceil$ means the round up 
and $\lceil \bfa\rceil=(\lceil \ab(1)\rceil, \cdots, \lceil \ab(n)\rceil)$. 
\item The module $\calT(\bfa)$ is a divisorial ideal and any divisorial ideal of $R$ takes this form. 
Therefore, we can identify each $\bfa\in\ZZ^n$ with the divisorial ideal $\calT(\bfa)$.
In particular, $\calT(\chi_{[n]})$ is isomorphic to the canonical module of $R$ (\cite{Stanley1978Hilbert}).
\item It is known that the isomorphic classes of divisorial ideals of $R$ one-to-one correspond to the elements of the divisor class group $\Cl(R)$ of $R$. 
We see that for $\bfa, \bfa^\prime\in\ZZ^n$, $\calT(\bfa)\cong \calT(\bfa^\prime)$ if and only if there exists ${\bf y}\in \ZZ^d$ such that $\ab=\ab^\prime+A^T_V\yb$. 
Thus, we have $\Cl(R)\cong\ZZ^n/\im(A^T_V)$.
In particular, we have 
\[
\Cl(R) \cong \ZZ^{n-r} \oplus \ZZ/s_1\ZZ \oplus \cdots \oplus \ZZ/s_r\ZZ,
\]
where $r:=\rank A_V $ and $s_1,\ldots,s_r$ are positive integers appearing in the diagonal of the Smith normal form of $A_V$.
\end{itemize}
The integers $s_1,\ldots,s_m$ are called the \textit{invariant factors} of $A_V$.
It is known that $s_i=g_i(A_V)/g_{i-1}(A_V)$ (see, e.g., \cite{newman1972integral}).
Here, for a matrix with integer entries $M$ and $i\in \ZZ_{\ge 0}$, $g_i(M)$ denotes the greatest common divisor of all $i\times i$ minors of $M$.
We let $g_0(M)=1$, and $g_i(M)=0$ if $i>\rank M$.
\medskip

We are interested in a divisorial ideal called \textit{conic}. 
\begin{definition}[{see, e.g., \cite[Section~3]{bruns2003divisorial}}]
\label{def_conic}
We say that a divisorial ideal $\calT(\bfa)$ is \emph{conic} if there exists $\bfx\in\RR^d$ such that $\bfa=\lceil A^T_V\xb\rceil$.
Let $\scrC(R)$ be the set of isomorphic classes of conic divisorial ideals of $R$. 
\end{definition}
Although the definition given here looks different from the one in the introduction, the two are equivalent (see \cite{bruns2005conic}).

We give a description of a region in $\Cl(R)$ corresponding to the classes of conic divisorial ideals of $R$.
The following argument already appears in \cite[Section~1]{bruns2003divisorial}.

Since $\Cl(R)\cong \ZZ^n/\im(A_V^T)$, there is an exact sequence of $\ZZ$-modules:
 \begin{equation}\label{cl_seq}
 0 \longrightarrow \im(A^T_V) \lhook\joinrel\longrightarrow \ZZ^n \xlongrightarrow{\pi}\Cl(R) \longrightarrow 0. 
\end{equation}
Let $\Cl(R)_\RR:=\Cl(R)\otimes_\ZZ \RR$ and $\pi_{\RR} : \RR^n \to \Cl(R)_\RR$ be the $\RR$-linear extension of $\pi$.
Let $\Cl(R)_{\mathrm{tor}}$ (resp. $\Cl(R)_{\mathrm{fre}}$) be the torsion subgroup (resp. the free subgroup) of $\Cl(R)$.
Note that $\Cl(R)_{\mathrm{fre}}$ is naturally regarded as a lattice in $\Cl(R)_\RR$.
In addition, for each $i=1,\ldots,n$, let $\beta_i:=\pi(\chi_i)$ and $\widetilde{\beta_i}:=\pi_\RR(\chi_i)\in \Cl(R)_{\mathrm{fre}}$.
We refer to the elements $\beta_1,\ldots,\beta_n$ as \textit{weights} of $R$. Note that these weights depend on the choice of $\pi$.

We set $\calB:=\{\widetilde{\beta_1},\ldots,\widetilde{\beta_n}\}$ and 
\[
\calW(\calB):=\pi_\RR([0,1)^n)=\left\{\sum_{i=1}^\ell \alpha_i\widetilde{\beta_i} : 0\le \alpha_i <1 \; \forall i\in [n]\right\}\subset \Cl(R)_\RR.
\]
For $\ab\in \ZZ^n$ such that there exists $\yb\in \RR^d$ with $\ab=\ceil{ A^T_V\yb}$, we have $\ab- A^T_V\yb=\ceil{ A^T_V\yb}- A^T_V\yb\in [0,1)^n$.
Therefore, the divisorial ideal $\calT(\ab)$ is conic if and only if $(\ab-\im(A^T_V))\cap [0,1)^n\neq \emptyset$, and hence we get the following proposition:
\begin{prop}[{cf. \cite[Corollary~1.5 (a)]{bruns2005conic}}]\label{prop:coniccorres}
    Under the above notation, there exists a bijection between $\scrC(R)$ and $(\calW(\calB)\cap \Cl(R)_{\mathrm{fre}})\times \Cl(R)_{\mathrm{tor}}$.
\end{prop}

\bigskip

\section{Conic divisorial ideals via matroid theory}\label{sec:conictoric}
In this section, we introduce the necessary notation from matroid theory and establish our results on conic divisorial ideals of toric rings.

\subsection{Preliminaries on matroid theory}

This subsection is devoted to recalling matroid theory.
We refer the reader to e.g., \cite{Bjorner1999oriented,oxley2006matroid} for the introduction to (oriented) matroid theory.

Let $M$ be a matroid on the ground set $E$ with the set of independent sets $\fkI$, the set of bases $\fkB$ and the rank function $\rk:2^E \to \ZZ_{\geq 0}$.
We collect some fundamental objects associated with matroids:
\begin{itemize}
    \item For a subset $A\subset E$, let $\mathrm{cl}(A)$ be the closure of $A$, that is, $\mathrm{cl}(A):=\{e\in E : \rk(A\cup\set{e})=\rk(A)\}$. We say that a subset $F \subset E$ is a \textit{flat} if $\mathrm{cl}(F) = F$.
    \item A subset $S\subset E$ is called \textit{spanning} if it contains a basis of $M$.
    \item A subset $C\subset E$ is called a \textit{circuit} if it is a minimal dependent set of $M$. We denote the set of circuits of $M$ by $\fkC(M)$.
    \item The \textit{dual matroid} $M^* := (E,\fkI^*)$ of $M$ is
defined via
\begin{align*}
    \fkI^{*} \coloneqq \left\{ J \subseteq E : E \setminus J \ \text{is a
spanning set of } M \right\}
\end{align*}
with the rank function $\rk^*(S):=\rk(E\setminus S)+\#S-\rk(S)$ for $S\subset E$.
    \item A \textit{hyperplane} of $M$ is a flat of rank $\rk(M)-1$. It is known that a subset $S\subset E$ is a circuit of $M$ if and only if $E\setminus S$ is a hyperplane of $M^*$ (\cite[Proposition~2.1.6]{oxley2006matroid}).
\end{itemize}

Throughout this paper, we consider only matroids that are representable over $\RR$. 
Given a finite set of vectors $V:=\{\vb_1,\ldots,\vb_n\}\subset \ZZ^d$, we write $M(V)$ for the matroid on $V$ defined by linear independence over $\RR$.
Then we have $\rk(M(V))=r:=\rank A_V$.

Fix a basis $B=\{\bb_1,\ldots,\bb_{n-r}\}\subset \ZZ^n$ of $\ker(A_V)$ and let $\overline{V}:=\{\overline{\vb}_1,\ldots,\overline{\vb}_n\}\subset \ZZ^{n-r}$ denote the set of column vectors of the matrix $A^T_B$. 
Then $M^*(V)$ is isomorphic to $M(\overline{V})$ as matroids via the correspondence $\vb_i \mapsto \overline{\vb}_i$ for each $i$.

\medskip

Let $\calZ(V):=A_V[0,1]^n=\set{\sum_{i=1}^n a_i\vb_i : 0\le a_i \le1 \; \forall i\in [n]}\subset \RR^d$ be the zonotope generated by $\vb_1,\ldots,\vb_n$.
According to oriented matroid theory, there is a bijection between the flats of \(M(V)\) and the pairs of parallel opposite faces of \(Z(V)\).
Under this correspondence, hyperplanes of \(M(V)\) correspond exactly to pairs of parallel opposite facets of \(Z(V)\).
More precisely, we have the following lemma:

\begin{lem}[{cf. \cite[Proposition~2.2.2]{Bjorner1999oriented}}]\label{lem:facetzono}
Let $H\subset V$ be a hyperplane of $M(V)$.
Then there exists $\wb^*\in \Hom_\ZZ(\ZZ^d,\ZZ)$ with $\{\vb_i\in V : \wb^*(\vb_i)=0\}=H$, and
\begin{align*}
F^+_H&:=\set{\sum_{\substack{\vb_i\in V\\ \wb^*(\vb_i)> 0}} \vb_i+\sum_{\vb_j\in H } a_j\vb_j : a_j\in [0,1] }\quad \text{ and }\\
F^-_H&:=\set{\sum_{\substack{\vb_i\in V\\ \wb^*(\vb_i)< 0}} \vb_i+\sum_{\vb_j\in H } a_j\vb_j : a_j\in [0,1] }
\end{align*}
are facets of $\calZ(V)$.
Conversely, all facets of $\calZ(V)$ are obtained in this way.
\end{lem}

\subsection{Description of conic divisorial ideals of toric rings}
In this section, we use the notation introduced in Section~\ref{sec:pre}: recall that $V:=\{\vb_1,\ldots,\vb_n\}\subset \ZZ^d$ is a finite set of primitive integer vectors and assume that the rational polyhedral cone $C(V):=\RR_{\ge 0} V$ is strongly convex.
We set $R:=\kk[C^\vee(V)\cap \ZZ^d]$.
Take an exact sequence (\ref{cl_seq}) and let $\beta_i:=\pi(\chi_i)$ and $\widetilde{\beta_i}:=\pi_\RR(\chi_i)\in \Cl(R)_{\mathrm{fre}}$ for each $i=1,\ldots,n$.
In addition, we set $\calB:=\{\widetilde{\beta_1},\ldots,\widetilde{\beta_n}\}$.

Applying \(\Hom_{\ZZ}(-,\ZZ)\) to the exact sequence (\ref{cl_seq}), we obtain the exact sequence
\[
0 \longrightarrow \Hom_{\ZZ}(\Cl(R),\ZZ)
   \longrightarrow \Hom_\ZZ(\ZZ^n,\ZZ)
   \longrightarrow \Hom_{\ZZ}(\im(A^T_V),\ZZ).
\]
The last term is regarded as a subgroup of \(\Hom_\ZZ(\ZZ^d,\ZZ)\) via the surjection $A_V^T : \ZZ^d \to \im(A_V^T)$.
Under the standard identifications $\Hom_{\ZZ}(\ZZ^i,\ZZ)\cong \ZZ^i$,
the composition
\[
\Hom_\ZZ(\ZZ^n,\ZZ)
\longrightarrow
\Hom_{\ZZ}(\im(A^T_V),\ZZ)
\longrightarrow
\Hom_\ZZ(\ZZ^d,\ZZ)
\]
is identified with \(A_V\). Therefore, we have $\ker(A_V)\cong \Hom_{\ZZ}(\Cl(R),\ZZ)\cong \Hom_\ZZ(\Cl(R)_{\mathrm{fre}},\ZZ)$.
In particular, the matroid $M(\calB)$ is isomorphic to $M^*(V)$ under the correspondence $\vb_i \mapsto \widetilde{\beta_i}$.

For a circuit $C:=\{\vb_{i_1},\ldots,\vb_{i_k}\}$ of $M(V)$, there exists a primitive vector $\ab_C\in \ZZ^n$ satisfying $A_V\ab_C={\bf 0}$ (that is, $\ab_C\in \ker(A_V)$) such that $\ab_C(j)\neq 0$ for $j\in \{i_1,\ldots,i_k\}$ and $\ab_C(j)=0$ otherwise.
Such a vector $\ab_C$ is unique up to sign.
Since $\ker(A_V)\cong \Hom_\ZZ(\Cl(R)_{\mathrm{fre}},\ZZ)$, the vector $\ab_C$ corresponds to an element of $\Hom_\ZZ(\Cl(R)_{\mathrm{fre}},\ZZ)$, denoted by $\ab_C^*$.
Then we can see that $\ab_C^*(\widetilde{\beta_i})=\ab_C(i)$ for each $i=1,\ldots,n$.

From this observation, we can give an inequality description of $\calW(\calB)$:
\begin{thm}\label{thm:conicdesc}
    We have 
    \[
    \calW(\calB)=\bigcap_{C\in \fkC(M(V))}\set{\zb \in \Cl(R)_\RR : -|\ab_C^-|< \ab^*_C(\zb)< |\ab^+_C|}.
    \]
\end{thm}
\begin{proof}
    Since the relative interior $\inte(\calZ(\calB))$ of $\calZ(\calB)$ coincides with $\calW(\calB)$ (\cite[In the proof of Lemma~3.2]{matsushita2022conic}), it is enough to show that
    \[
    \calZ(\calB)=\bigcap_{C\in \fkC(M(V))}\set{\zb \in \Cl(R)_\RR : -|\ab_C^-|\le \ab^*_C(\zb)\le |\ab^+_C|}.
    \]
    Let $H:=\{\widetilde{\beta_{j_1}},\ldots,\widetilde{\beta_{j_k}}\}\subset \calB$ be a hyperplane of $M(\calB)$.
    Since $M(\calB)$ is isomorphic to $M^*(V)$, $C:=V\setminus \{\vb_{j_1},\ldots,\vb_{j_k}\}$ is a circuit of $M(V)$.
    From Lemma~\ref{lem:facetzono}, for any $\xb\in F^+_H$ (resp. $\xb\in F^-_H$), we have 
    \[
    \ab^*_C(\xb)=\sum_{\substack{\widetilde{\beta}\in \calB\\ \ab^*_C(\widetilde{\beta})> 0}}\ab^*_C(\widetilde{\beta})=|\ab_C^+|
    \quad (\text{resp. }\ab^*_C(\xb)=\sum_{\substack{\widetilde{\beta}\in \calB\\ \ab^*_C(\widetilde{\beta})< 0}}\ab^*_C(\widetilde{\beta})=-|\ab_C^-|).
    \]
    Thus, we can see that $\calZ(\calB)\subset \set{\zb \in \Cl(R)_\RR : -|\ab_C^-|\le \ab^*_C(\zb)\le |\ab^+_C|}$ and that $\{\zb \in \Cl(R)_\RR : \ab^*_C(\zb)= |\ab^+_C|\}$ (resp. $\{\zb \in \Cl(R)_\RR : \ab^*_C(\zb)= -|\ab^-_C|\}$) is the supporting hyperplane of $F^+_H$ (resp. $F^-_H$).
    This shows our assertion.
\end{proof}

\begin{ex}
    Let
    \begin{align*}
    \vb_1:&=(1,0,0,0), &\vb_2:&=(0,1,0,0), &\vb_3:&=(0,0,1,0),\\ 
    \vb_4:&=(-1,0,0,1), &\vb_5:
    &=(0,-1,0,1), &\vb_6:&=(0,0,-1,1).
    \end{align*}
    Then we have $\Cl(R)= \Cl(R)_{\mathrm{fre}}\cong \ZZ^2$ where $R=\kk[C^\vee(V)\cap \ZZ^4]$. Moreover, let
    \begin{align*}
    \beta_1=\beta_4=(1,0), \quad \beta_2=\beta_5=(0,1), \quad \beta_3=\beta_6=(-1,-1),
    \end{align*}
    $V:=\{\vb_1,\ldots,\vb_6\}$ and $\calB:=\{\beta_1,\ldots,\beta_6\}$.
    Then we get the exact sequence
    \begin{equation*}
 0 \longrightarrow \im(A^T_V) \lhook\joinrel\longrightarrow \ZZ^6 \xlongrightarrow{A_\calB}\ZZ^2 \longrightarrow 0. 
\end{equation*}
The matroid $M(V)$ has three circuits $C_1:=\{\vb_1,\vb_2,\vb_4,\vb_5\}$, $C_2:=\{\vb_2,\vb_3,\vb_5,\vb_6\}$ and $C_3:=\{\vb_1,\vb_3,\vb_4,\vb_6\}$, and we can see that 
\begin{align*}
    \ab_{C_1}=(1,-1,0,1,-1,0), \quad \ab_{C_2}=(0,1,-1,0,1,-1), \quad \ab_{C_3}=(1,0,-1,1,0,-1).
\end{align*}
Moreover, we can see that for $\zb=(z_1,z_2)=z_1\beta_1+z_2\beta_2\in \ZZ^2$, 
\begin{align*}
    \ab^*_{C_1}(\zb)=z_1-z_2, \quad \ab^*_{C_2}(\zb)=z_2, \quad \ab^*_{C_3}(\zb)=z_1.
\end{align*}
Therefore, we have 
\[
    \calW(\calB)=\set{\zb \in \Cl(R)_\RR : \hspace{-0.1cm}
\begin{array}{ccc}
-2 < z_1-z_2 < 2, \vspace{0.1cm}\\
 -2 < z_2 < 2, \vspace{0.1cm}\\
-2 < z_1 < 2
\end{array}}.
    \]
\end{ex}

\medskip

In the remainder of this subsection, we discuss the number of isomorphic classes of conic divisorial ideals.

To compute the number $\#\scrC(R)$, we use the Ehrhart polynomial of $\calZ(\calB)$.
For a lattice polytope $P \subset \RR^d$ with $r:=\dim P$, the \textit{Ehrhart polynomial} $\ehr_P(t)$ is defined by
\[
\ehr_P(t) := \#(tP \cap \ZZ^d)
\]
for positive integers $t$.
It is well known that $\ehr_P(t)$ is in fact a polynomial in $t$ of degree $r$. Moreover, by Ehrhart reciprocity, $(-1)^r\ehr_P(-1)$ is equal to the number of lattice points in the interior of $P$.
For background on Ehrhart polynomials, see, e.g., \cite{beck2007computing}.

In \cite{bach2024acyclotopes}, the zonotope $\calZ(\calB)$ is called the \textit{lattice Gale zonotope} of $V$ and its Ehrhart polynomial is calculated as follows:
\begin{thm}[{\cite[Theorem~5.1]{bach2024acyclotopes}}]\label{thm:ehrlatGalezono}
Let $r:=\rank A_V$. Then we have
\begin{align*}
 \ehr_{\calZ(\calB)}(t)=
 \sum_{S\subset V}\frac{g_r(A_S)}{g_r(A_V)} t^{n-\#S}.
\end{align*}
\end{thm}

\begin{cor}\label{cor:numberconic}
    We have
    \[
    \#\scrC(R)=\sum_{S\subset V} (-1)^{\#S-r}g_{r}(A_S).
    \]
\end{cor}
\begin{proof}
    We write $\Cl(R)\cong \ZZ^{n-r} \oplus \ZZ/s_1\ZZ \oplus \cdots \oplus \ZZ/s_r\ZZ$, then $\#\Cl(R)_{\mathrm{tor}}=s_1\cdots s_r=g_r(A_V)$.
    Thus, it follows from Proposition~\ref{prop:coniccorres}, Theorem~\ref{thm:ehrlatGalezono} and the fact $\inte(\calZ(\calB))=\calW(\calB)$ that
    \begin{align*}
    \#\scrC(R)&=\#(\inte(\calZ(\calB))\cap \Cl(R)_{\mathrm{fre}})\cdot\#\Cl(R)_{\mathrm{tor}} \\
    &=((-1)^{n-r}\ehr_{\calZ(\calB)}(-1))\cdot g_r(A_V)=\sum_{S\subset V} (-1)^{\#S-r}g_{r}(A_S).
    \end{align*}
\end{proof}

\begin{rem}\label{rem:number}
    In fact, the above corollary also admits another proof using the multiplicity Tutte polynomial.

    For $\yb,\yb' \in \RR^d$, we denote $\yb \sim \yb'$ if $\calT(\ceil{A_V^T\yb})=\calT(\ceil{A_V^T\yb'})$.
This defines an equivalence relation on $\RR^d$, and hence induces a partition $\mathscr{S}$ of $\RR^d$.
An equivalence class (a cell of $\mathscr{S}$) containing $\yb \in \RR^d$ is of the form 
\[
\calD(\yb):=\{\xb\in \RR^d : \ceil{A_V^T\yb}(i)-1< A_V^T\xb(i)\le \ceil{A_V^T\yb}(i) \text{ for any }i\in [n]\},
\]
and it is full-dimensional since $C(V)$ is strongly convex.
Moreover, two conic divisorial ideals $\calT(\ceil{\nu(\yb)})$ and $ \calT(\ceil{\nu(\yb')})$ are isomorphic if and only if $\calD(\yb)=\calD(\yb')+\zb$ for some $\zb\in \ZZ^d$.
Thus, a full dimensional cell of the decomposition $\mathscr{T}:=\mathscr{S}/\ZZ^d$ of the $d$-th dimensional torus $\TT^d:=\RR^d/\ZZ^d$ one-to-one corresponds to an isomorphic class of conic divisorial ideals.
This argument already appears in \cite{bruns2005conic}.
In particular, the number of isomorphic classes of conic divisorial ideals of $R$ is equal to that of connected components of $\TT\setminus \bigcup_{i=1}^n\overline{H}_{\vb_i}$, where $\overline{H}_{\vb_i}$ denotes the image of $H_{\vb_i}:=\{\xb \in \RR^d : \langle \vb_i,\xb \rangle = 0\}$ under the quotient map $\RR^d\to \TT^d$.
The finite set $\{\overline{H}_{\vb_1},\ldots,\overline{H}_{\vb_n}\}$ is called a \textit{toric hyperplane arrangement}.

On the other hand, the matroid $M(V)$ carries the structure of a multiplicity (arithmetic) matroid, and its multiplicity Tutte polynomial is given as follows:
\[
M_V(x,y):=\sum_{S\subseteq V} g_{\rk(S)}(A_S) (x-1)^{\rk(V)-\rk(S)} (y-1)^{\#S-\rk(S)}.
\]
For background on multiplicity matroids and multiplicity Tutte polynomials, see, e.g., \cite{Michele2013arithmetic,Luca2012tutte}.
It is known from \cite[Corollary~5.16]{Luca2012tutte} that the number of connected components of $\TT\setminus \bigcup_{i=1}^n\overline{H}_{\vb_i}$ is given by $M_V(1,0)=\sum_{S\subset V} (-1)^{\#S-r}g_r(A_S)$, which agrees with our result.
\end{rem}

\bigskip

\section{Applications to toric rings associated with root systems}\label{sec:4}

In this section, we study the toric rings arising from root systems.

\subsection{Preliminaries on signed posets and their toric rings}\label{subsec:signed poset}

First, we recall the necessary definitions and notation for signed posets and signed graphs.

Let $\Phi\subset \RR^d$ be a root system.
For $P \subset \Phi$, the \textit{positive linear closure},
$\overline{P}$, is the set of $\alpha\in \Phi$ which are non-negative linear combinations of elements of $P$, i.e. $\overline{P}=C(P)\cap \Phi$.
A \textit{$\Phi$-poset} is a subset $P \subset \Phi$ such that (i) if $\alpha \in P$, then $-\alpha \notin P$, and (ii) $P = \overline{P}$.

For a $\Phi$-poset $P \subset \Phi$, we are interested only in the primitive ray structure of the cone $C(P)$.
Accordingly, we define $\widetilde{P}$ as the set of primitive generators of the extremal rays of $C(P)$.
In other words, one first removes all redundant generators from $P$, and then replaces each remaining generator by the primitive lattice vector on the same ray.
Although $\widetilde{P}$ may fail to be a subset of $\Phi$, this is precisely the level of generality relevant to the semigroup rings studied here.

We define the toric ring $R_P$ of a $\Phi$-poset $P$ as the toric ring determined by $\widetilde{P}$, that is,
\[
R_P:=\kk[C^\vee(\widetilde{P})\cap \ZZ^d].
\]

Let $\Phi_{A_{d-1}}$ and $\Phi_{B_d}$ be the root systems of type $A$ and $B$, respectively. More precisely, we set 
\begin{align*}
\Phi_{A_{d-1}}&:= \{\chi_i-\chi_j : i,j\in [d]\}\subset \RR^d  \text{ and }\\ 
\Phi_{B_d}&:=\{\pm\chi_i\pm\chi_j : i,j\in [d]\}\cup \{\pm \chi_i : i\in [d]\}\subset \RR^d.
\end{align*}
In this paper, we mainly consider the root system of type $B$.
From the perspective adopted here, this is the most general classical case: after passing to primitive ray generators, type $C$ yields no essentially new data, while type $A$ and type $D$ arise as subconfigurations of type $B$.
Thus, it is enough for our purposes to work throughout with $\Phi_{B_d}$.

When $\Phi=\Phi_{B_d}$, a $\Phi$-poset $P$ is usually called a \textit{signed poset}.
In this case, $P$ admits a convenient combinatorial description by means of an oriented signed graph, which we regard as its Hasse diagram.
To make this correspondence precise, we briefly recall the notion of an oriented signed graph.
We refer the reader to \cite{zaslavsky1982signed,zaslavsky1991orientation} for an introduction to signed graphs.

Graphs are not assumed to be simple and may contain loops (edges whose two endpoints coincide), multiple edges, halfedges (edges with exactly one endpoint) and loose edges (edges with no endpoints).
In this paper, however, we restrict attention to graphs with neither loops nor loose edges.
This is sufficient for our purposes, since signed graphs corresponding to signed posets of $\Phi_{B_d}$ have only (multiple) edges and halfedges.

A \textit{signed graph} $\Sigma$ is a pair $(\Gamma,\sigma)$ consisting of a graph $\Gamma$ on vertex set $V(\Gamma)$ with edge set $E(\Gamma)$ and a map $\sigma$ assigning a sign to each edge of $\Gamma$ except for halfedges.

An \textit{orientation} of a signed graph $\Sigma=(\Gamma,\sigma)$ is a map $\tau$ from the set of vertex-edge incidences to $\{\pm\}$ such that $\sigma(e)=-\tau(e,u)\tau(e,v)$ for every edge $e=\{u,v\}$.
Equivalently, one may begin with an unsigned graph $\Gamma=(V(\Gamma),E(\Gamma))$ and choose a \textit{bidirection} $\tau$ on $\Gamma$, that is, a map from the set of vertex-edge incidences to $\{\pm\}$. Then, for each edge $e=\{u,v\}$, setting
$\sigma(e)=-\tau(e,u)\tau(e,v)$ produces a signed graph $\Sigma=(\Gamma,\sigma)$ together with an orientation $\tau$. Thus, oriented signed graphs and bidirected graphs may be viewed as equivalent objects.
One may interpret $\tau$ as follows. If $\tau(e,v)=+$ (resp. $\tau(e,v)=-$), then the edge $e$ is directed into (resp. away from) the vertex $v$, that is, the arrow at the incidence $(e,v)$ points toward (resp. away from) $v$.

For an oriented signed graph $\Sigma$ with an orientation $\tau$, we define the \textit{incidence matrix}
$A_\Sigma\in\ZZ^{V(\Gamma)\times E(\Gamma)}$ by 
\begin{align*}
 A_\Sigma(v,e):=\begin{cases}
          0 & \text{ if $v$ and $e$ are not incident,}\\
          +1 & \text{ if $e$ enters $v$, that is, }\tau(v,e)=+\,,\\
          -1 & \text{ if $e$ exits $v$, that is, }\tau(v,e)=-.
         \end{cases}
\end{align*}

We now define the \textit{Hasse diagram} $\calH_P=(\Gamma_P,\sigma_P)$ of a signed poset $P\subset \Phi_{B_d}$ to be the oriented signed graph satisfying $A_{\calH_P}=A_{\widetilde{P}}$.
More precisely, $\calH_P$ is the bidirected graph on the vertex set $V(\Gamma_P)=[d]$ with the edge set $E(\Gamma_P)=\widetilde{P}$, where each $e\in\widetilde{P}$ is regarded as an edge as follows:
\begin{itemize}
    \item if $e=\pm \chi_i$ for some $i$, then $e$ is a halfedge incident to $i$;
    \item if $e=\pm \chi_i\pm \chi_j$ for some $i\neq j$, then $e$ is an edge joining $i$ and $j$.
\end{itemize}
The bidirection of an edge $e$ is chosen so that the column of $A_{\mathcal H_P}$ corresponding to $e$ is exactly the vector $e\in \widetilde{P}$.
Thus, we have $e=\tau(e,i)\chi_i$ (resp. $e=\tau(e,i)\chi_i+\tau(e,j)\chi_j$) if $e=\{i\}$ (resp. $e=\{i,j\}$).

\begin{ex}
Let $P\subset \Phi_{B_4}$ be a signed poset with 
\[
\widetilde{P}=\{e_1:=\chi_1,\; e_2:=\chi_2-\chi_1,\;
e_3:=\chi_2+\chi_3,\; e_4:=\chi_2-\chi_3,\; e_5:=-\chi_3-\chi_4,\; e_6:=\chi_1-\chi_4\}.
\]
The Hasse diagram of $P$ is depicted as in Figure~\ref{Figure:ex}, and we have 
\[
A_{\calH_P}=\begin{pmatrix}
1 &-1 &0 &0 &0 &1 \\
0 &1 &1 &1 &0 &0 \\
0 &0 &1 &-1 &-1 &0  \\
0 &0 &0 &0 &-1 &-1
\end{pmatrix}.
\]
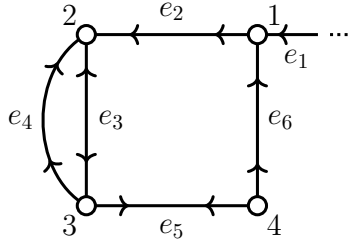
\begin{figure}[H]

\centering
\scalebox{0.8}{
\begin{tikzpicture}[line width=0.05cm]
\def\r{2};
\coordinate (N0) at ($(0,0)+({\r*cos(45)},{\r*sin(45)})$); 
\coordinate (N1) at ($(0,0)+({\r*cos(135)},{\r*sin(135)})$); 
\coordinate (N2) at ($(0,0)+({\r*cos(225)},{\r*sin(225)})$); 
\coordinate (N3) at ($(0,0)+({\r*cos(315)},{\r*sin(315)})$); 

\coordinate (P1) at ($(N0)+(1.3,0)$);

\draw[
  postaction={
    decorate,
    decoration={
      markings,
      mark=at position 0.75 with {\arrow{>}},
      mark=at position 0.25 with {\arrow{>}}
    }
  }
] (N0) -- (N1);
\draw[
  postaction={
    decorate,
    decoration={
      markings,
      mark=at position 0.75 with {\arrow{>}},
      mark=at position 0.25 with {\arrow{<}}
    }
  }
] (N1) -- (N2);
\draw[
  postaction={
    decorate,
    decoration={
      markings,
      mark=at position 0.75 with {\arrow{<}},
      mark=at position 0.25 with {\arrow{>}}
    }
  }
] (N2) -- (N3);
\draw[
  postaction={
    decorate,
    decoration={
      markings,
      mark=at position 0.75 with {\arrow{>}},
      mark=at position 0.25 with {\arrow{>}}
    }
  }
] (N3) -- (N0);
\draw[
  postaction={
    decorate,
    decoration={
      markings,
      mark=at position 0.75 with {\arrow{<}},
      mark=at position 0.25 with {\arrow{<}}
    }
  }
] (N1) to [out=210,in=150] (N2);
\draw[
  postaction={
    decorate,
    decoration={
      markings,
      mark=at position 0.5 with {\arrow{<}}
    }
  }
] (N0) -- ($(N0)+(1,0)$);
\draw[dotted] ($(N0)+(1.2,0)$) -- ($(N0)+(1.5,0)$);

\draw [line width=0.05cm, fill=white] (N0) circle [radius=0.15] node[above right] {\Large $1$};
\draw [line width=0.05cm, fill=white] (N1) circle [radius=0.15] node[above left] {\Large $2$};
\draw [line width=0.05cm, fill=white] (N2) circle [radius=0.15] node[below left] {\Large $3$};
\draw [line width=0.05cm, fill=white] (N3) circle [radius=0.15] node[below right] {\Large $4$};
\def\l{1.8};
\node[] at ($(0,0)+({\l*cos(90)},{\l*sin(90)})$) {\Large $e_2$};
\node[] at ($(0,0)+({\l*cos(270)},{\l*sin(270)})$) {\Large $e_5$};
\node[] at ($(0,0)+({\l*cos(0)},{\l*sin(0)})$) {\Large $e_6$};
\node[] at (-1,0) {\Large $e_3$};
\node[] at ($(N0)+(0.65,-0.4)$) {\Large $e_1$};
\node[] at (-2.5,0) {\Large $e_4$};
\end{tikzpicture}}
\caption{The Hasse diagram $\calH_P$}
\label{Figure:ex}

\end{figure}

\end{ex}

Let $\Sigma=(\Gamma,\sigma)$ be a signed graph, and let 
For $F\subset E(\Gamma)$, we write $\Sigma(F)$ for the \textit{signed subgraph} whose underlying graph is $\Gamma(F):=(\bigcup_{e\in F} e, F)$ and whose signature is the restriction of $\sigma$ to $F$.

We recall the following notions from signed graph theory:

\begin{itemize}
   \item A \textit{walk} is a sequence $(e_1,v_1,e_2,\dots,v_{\ell-1},e_\ell)$ of vertices $v_i$ and edges $e_i$, such that each vertex $v_i$ is incident with both $e_i$ and $e_{i+1}$. (the edges and vertices in the sequences are not necessarily pairwise distinct).
   We call $\ell$ the \textit{length} of $W$.
   We denote by $V(W):=\{v_1,\ldots,v_{\ell-1}\}$ and $E(W):=\{e_1,\ldots,e_\ell\}$ the multisets of vertices and edges appearing in $W$, respectively.
   \item We say that a walk $W=(e_1,v_1,e_2,\dots,v_{\ell-1},e_\ell)$ is \textit{positive} (resp. \textit{negative}) if $\sigma(e_1)\cdots \sigma(e_\ell)=+$ (resp. $\sigma(e_1)\cdots \sigma(e_\ell)=-$).
   \item A \textit{path} is a walk in which all the edges and vertices are pairwise distinct. 
   \item A signed graph $\Sigma$ is \textit{connected} if there exists a path between any two nodes.
   \item We say that a walk $W=(e_1,v_1,e_2,\dots,v_{\ell-1},e_\ell)$ is \textit{closed} if $e_1$ and $e_\ell$ are incident with the same vertex $v_\ell$ and $v_\ell\neq v_1, v_{\ell-1}$.
   We call $v_\ell$ the \textit{base vertex} of $W$.
   \item A \textit{circle} is a closed path.
    \item A \textit{signed tree} is a connected signed graph with no circles or halfedges.
    \item A \textit{signed halfedge-tree} is a connected signed graph with no circles and a single halfedge.
    \item A \textit{signed pseudo-tree} is a connected signed graph with no halfedges that contains a single negative circle.
    \item A \textit{signed pseudo-forest} is a signed graph whose connected components are signed trees, signed halfedge-trees, or signed pseudo-trees.
    \item A \textit{circuit} is a signed subgraph with an inclusion minimal set of edges that is not a pseudo-forest.
    We denote the set of circuits of $\Sigma$ by $\fkC(\Sigma)$.
\end{itemize}

We can see that every circuit of $\calH_P$ is one of the following four signed graphs:
\begin{itemize}
  \item[($\calC_{T_1}$)] a positive circle (see Figure~\ref{fig:circuit1});
  \item[($\calC_{T_2}$)] two negative circles joined by a path (see Figure~\ref{fig:circuit2});
  \item[($\calC_{T_3}$)] a negative circle joined to a halfedge by a path (see Figure~\ref{fig:circuit3});
  \item[($\calC_{T_4}$)] two halfedges joined by a path (see Figure~\ref{fig:circuit4}).
\end{itemize}

\begin{figure}[H]
    \centering
    \begin{minipage}{0.36\linewidth}
        \centering
        \scalebox{0.8}{
            \begin{tikzpicture}[line width=0.05cm]
            \def\r{1.5};
\coordinate (N0) at ($(0,0)+({\r*cos(18)},{\r*sin(18)})$); 
\coordinate (N1) at ($(0,0)+({\r*cos(90)},{\r*sin(90)})$); 
\coordinate (N2) at ($(0,0)+({\r*cos(162)},{\r*sin(162)})$); 
\coordinate (N3) at ($(0,0)+({\r*cos(234)},{\r*sin(234)})$); 
\coordinate (N4) at ($(0,0)+({\r*cos(306)},{\r*sin(306)})$);
            \draw ($(0,0)+({\r*cos(-70)},{\r*sin(-70)})$) arc [start angle=-70, delta angle=320, radius=\r] ; 
            \draw[dotted] ($(0,0)+({\r*cos(-105)},{\r*sin(-105)})$) arc [start angle=-105, delta angle=30, radius=\r] ;
\draw [line width=0.05cm, fill=white] (N0) circle [radius=0.15] node[above right] {\Large $u_0$};
\draw [line width=0.05cm, fill=white] (N1) circle [radius=0.15] node[] at ($(N1)+(0,0.4)$) {\Large $u_1$};
\draw [line width=0.05cm, fill=white] (N2) circle [radius=0.15] node[above left] {\Large $u_2$};
\draw [line width=0.05cm, fill=white] (N3) circle [radius=0.15] node[below left] {\Large $u_3$};
\draw [line width=0.05cm, fill=white] (N4) circle [radius=0.15] node[below right] {\Large $u_{k-1}$};
\node[above right] at ($(0,0)+({\r*cos(54)},{\r*sin(54)})$) {\Large $f_1$};
\node[above left] at ($(0,0)+({\r*cos(126)},{\r*sin(126)})$) {\Large $f_2$};
\node[below left] at ($(0,0)+({\r*cos(198)},{\r*sin(198)})$) {\Large $f_3$};
\node[below right] at ($(0,0)+({\r*cos(342)},{\r*sin(342)})$) {\Large $f_k$};
            \end{tikzpicture}}
        \caption{\newline The circuit $\calC_{T_1}$}
        \label{fig:circuit1}
    \end{minipage}
    \begin{minipage}{0.63\linewidth}
        \centering
        \scalebox{0.8}{
        \begin{tikzpicture}[line width=0.05cm]
\def\r{1.5};
\coordinate (N0) at ($(0,0)+({\r*cos(18)},{\r*sin(18)})$); 
\coordinate (N1) at ($(0,0)+({\r*cos(90)},{\r*sin(90)})$); 
\coordinate (N2) at ($(0,0)+({\r*cos(162)},{\r*sin(162)})$); 
\coordinate (N3) at ($(0,0)+({\r*cos(234)},{\r*sin(234)})$); 
\coordinate (N4) at ($(0,0)+({\r*cos(306)},{\r*sin(306)})$);
\coordinate (C2) at (7,0); 

\coordinate (W0) at ($(C2)+({\r*cos(162)},{\r*sin(162)})$); 
\coordinate (X1) at ($(C2)+({\r*cos(234)},{\r*sin(234)})$);
\coordinate (X2) at ($(C2)+({\r*cos(306)},{\r*sin(306)})$);
\coordinate (X3) at ($(C2)+({\r*cos(18)},{\r*sin(18)})$);
\coordinate (WL) at ($(C2)+({\r*cos(90)},{\r*sin(90)})$); 

\coordinate (P1) at ($(N0)+(1.3,0)$);
\coordinate (P2) at ($(W0)+(-1.3,0)$);

            \draw ($(0,0)+({\r*cos(-70)},{\r*sin(-70)})$) arc [start angle=-70, delta angle=320, radius=\r] ; 
            \draw[dotted] ($(0,0)+({\r*cos(-105)},{\r*sin(-105)})$) arc [start angle=-105, delta angle=30, radius=\r] ;


\draw
  ($(C2)+({\r*cos(-70)},{\r*sin(-70)})$)
  arc[start angle=-70, delta angle=320, radius=\r];

\draw[dotted]
  ($(C2)+({\r*cos(-105)},{\r*sin(-105)})$)
  arc[start angle=-105, delta angle=30, radius=\r];

\draw (N0) -- (P1);
\draw (P2) -- (W0);
\draw (P2) -- ($(P2)+(-0.4,0)$);
\draw (P1) -- ($(P1)+(0.4,0)$);
\draw[dotted] ($(P1)+(0.5,0)$) -- ($(P2)+(-0.5,0)$);

\draw [line width=0.05cm, fill=white] (N0) circle [radius=0.15] node[below left] {\Large $u_0$};
\draw [line width=0.05cm, fill=white] (N1) circle [radius=0.15] node[] at ($(N1)+(0,0.4)$) {\Large $u_1$};
\draw [line width=0.05cm, fill=white] (N2) circle [radius=0.15] node[above left] {\Large $u_2$};
\draw [line width=0.05cm, fill=white] (N3) circle [radius=0.15] node[below left] {\Large $u_3$};
\draw [line width=0.05cm, fill=white] (N4) circle [radius=0.15] node[below right] {\Large $u_{k-1}$};
\node[above right] at ($(0,0)+({\r*cos(54)},{\r*sin(54)})$) {\Large $f_1$};
\node[above left] at ($(0,0)+({\r*cos(126)},{\r*sin(126)})$) {\Large $f_2$};
\node[below left] at ($(0,0)+({\r*cos(198)},{\r*sin(198)})$) {\Large $f_3$};
\node[below right] at ($(0,0)+({\r*cos(342)},{\r*sin(342)})$) {\Large $f_k$};
\node[above right] at ($(C2)+({\r*cos(54)},{\r*sin(54)})$) {\Large $g_2$};
\node[above left] at ($(C2)+({\r*cos(126)},{\r*sin(126)})$) {\Large $g_1$};
\node[below left] at ($(C2)+({\r*cos(198)},{\r*sin(198)})$) {\Large $g_l$};
\node[below right] at ($(C2)+({\r*cos(342)},{\r*sin(342)})$) {\Large $g_3$};

\node[] at ($(P1)+(-0.65,-0.4)$) {\Large $h_1$};
\node[] at ($(P2)+(0.65,-0.4)$) {\Large $h_m$};
\draw [line width=0.05cm, fill=white] (W0) circle [radius=0.15]
  node[below right] {\Large $w_0$};

\draw [line width=0.05cm, fill=white] (X1) circle [radius=0.15]
  node[below left] {\Large $w_{l-1}$};

\draw [line width=0.05cm, fill=white] (X2) circle [radius=0.15]
  node[below right] {\Large $w_3$};

\draw [line width=0.05cm, fill=white] (X3) circle [radius=0.15]
  node[above right] {\Large $w_2$};

\draw [line width=0.05cm, fill=white] (WL) circle [radius=0.15]
  node[] at ($(WL)+(0,0.4)$) {\Large $w_1$};

\draw [line width=0.05cm, fill=white] (P1) circle [radius=0.15] node[] at ($(P1)+(0,0.5)$) {\Large $\mu_1$};

\draw [line width=0.05cm, fill=white] (P2) circle [radius=0.15] node[] at ($(P2)+(0,0.5)$) {\Large $\mu_{m-1}$};

            \end{tikzpicture}}
        \caption{The circuit $\calC_{T_2}$}
        \label{fig:circuit2}
    \end{minipage}
\end{figure}

\begin{figure}[H]

    \begin{minipage}{0.49\linewidth}
        \centering
        \scalebox{0.8}{
            \begin{tikzpicture}[line width=0.05cm]
\def\r{1.5};
\coordinate (N0) at ($(0,0)+({\r*cos(18)},{\r*sin(18)})$); 
\coordinate (N1) at ($(0,0)+({\r*cos(90)},{\r*sin(90)})$); 
\coordinate (N2) at ($(0,0)+({\r*cos(162)},{\r*sin(162)})$); 
\coordinate (N3) at ($(0,0)+({\r*cos(234)},{\r*sin(234)})$); 
\coordinate (N4) at ($(0,0)+({\r*cos(306)},{\r*sin(306)})$);
\coordinate (C2) at (7,0); 

\coordinate (W0) at ($(C2)+({\r*cos(162)},{\r*sin(162)})$); 
\coordinate (X1) at ($(C2)+({\r*cos(234)},{\r*sin(234)})$);
\coordinate (X2) at ($(C2)+({\r*cos(306)},{\r*sin(306)})$);
\coordinate (X3) at ($(C2)+({\r*cos(18)},{\r*sin(18)})$);
\coordinate (WL) at ($(C2)+({\r*cos(90)},{\r*sin(90)})$); 

\coordinate (P1) at ($(N0)+(1.3,0)$);
\coordinate (P2) at ($(W0)+(-1.3,0)$);

            \draw ($(0,0)+({\r*cos(-70)},{\r*sin(-70)})$) arc [start angle=-70, delta angle=320, radius=\r] ; 
            \draw[dotted] ($(0,0)+({\r*cos(-105)},{\r*sin(-105)})$) arc [start angle=-105, delta angle=30, radius=\r] ;

\draw (N0) -- (P1);
\draw (P2) -- (W0);
\draw (P2) -- ($(P2)+(-0.4,0)$);
\draw (P1) -- ($(P1)+(0.4,0)$);
\draw[dotted] ($(P1)+(0.5,0)$) -- ($(P2)+(-0.5,0)$);
\draw (W0) -- ($(W0)+(0.6,0)$);
\draw[dotted] ($(W0)+(0.7,0)$) -- ($(W0)+(1.3,0)$);

\draw [line width=0.05cm, fill=white] (N0) circle [radius=0.15] node[below left] {\Large $u_0$};
\draw [line width=0.05cm, fill=white] (N1) circle [radius=0.15] node[] at ($(N1)+(0,0.4)$) {\Large $u_1$};
\draw [line width=0.05cm, fill=white] (N2) circle [radius=0.15] node[above left] {\Large $u_2$};
\draw [line width=0.05cm, fill=white] (N3) circle [radius=0.15] node[below left] {\Large $u_3$};
\draw [line width=0.05cm, fill=white] (N4) circle [radius=0.15] node[below right] {\Large $u_{k-1}$};
\node[above right] at ($(0,0)+({\r*cos(54)},{\r*sin(54)})$) {\Large $f_1$};
\node[above left] at ($(0,0)+({\r*cos(126)},{\r*sin(126)})$) {\Large $f_2$};
\node[below left] at ($(0,0)+({\r*cos(198)},{\r*sin(198)})$) {\Large $f_3$};
\node[below right] at ($(0,0)+({\r*cos(342)},{\r*sin(342)})$) {\Large $f_k$};

\node[] at ($(P1)+(-0.65,-0.4)$) {\Large $h_1$};
\node[] at ($(P2)+(0.65,-0.4)$) {\Large $h_m$};
\node[] at ($(W0)+(0.65,-0.4)$) {\Large $h_{m+1}$};
 \draw [line width=0.05cm, fill=white] (W0) circle [radius=0.15] node[] at ($(W0)+(0,0.4)$) {\Large $\mu_m$};

\draw [line width=0.05cm, fill=white] (P1) circle [radius=0.15] node[] at ($(P1)+(0,0.5)$) {\Large $\mu_1$};

\draw [line width=0.05cm, fill=white] (P2) circle [radius=0.15] node[] at ($(P2)+(0,0.5)$) {\Large $\mu_{m-1}$};
            \end{tikzpicture}}
        \caption{The circuit $\calC_{T_3}$}
        \label{fig:circuit3}
    \end{minipage}
    \begin{minipage}{0.5\linewidth}
        \centering
        \scalebox{0.8}{
            \begin{tikzpicture}[line width=0.05cm]
\def\r{1.5};
\coordinate (N0) at ($(0,0)+({\r*cos(18)},{\r*sin(18)})$); 
\coordinate (N1) at ($(0,0)+({\r*cos(90)},{\r*sin(90)})$); 
\coordinate (N2) at ($(0,0)+({\r*cos(162)},{\r*sin(162)})$); 
\coordinate (N3) at ($(0,0)+({\r*cos(234)},{\r*sin(234)})$); 
\coordinate (N4) at ($(0,0)+({\r*cos(306)},{\r*sin(306)})$);
\coordinate (C2) at (7,0); 

\coordinate (W0) at ($(C2)+({\r*cos(162)},{\r*sin(162)})$); 
\coordinate (X1) at ($(C2)+({\r*cos(234)},{\r*sin(234)})$);
\coordinate (X2) at ($(C2)+({\r*cos(306)},{\r*sin(306)})$);
\coordinate (X3) at ($(C2)+({\r*cos(18)},{\r*sin(18)})$);
\coordinate (WL) at ($(C2)+({\r*cos(90)},{\r*sin(90)})$); 

\coordinate (P1) at ($(N0)+(1.3,0)$);
\coordinate (P2) at ($(W0)+(-1.3,0)$);

\draw (N0) -- (P1);
\draw (P2) -- (W0);
\draw (P2) -- ($(P2)+(-0.4,0)$);
\draw (P1) -- ($(P1)+(0.4,0)$);
\draw[dotted] ($(P1)+(0.5,0)$) -- ($(P2)+(-0.5,0)$);
\draw (W0) -- ($(W0)+(0.6,0)$);
\draw (N0) -- ($(N0)+(-0.6,0)$);
\draw[dotted] ($(N0)+(-0.7,0)$) -- ($(N0)+(-1.3,0)$);
\draw[dotted] ($(W0)+(0.7,0)$) -- ($(W0)+(1.3,0)$);

\draw [line width=0.05cm, fill=white] (N0) circle [radius=0.15] node[] at ($(N0)+(0,0.4)$) {\Large $p_0$};

\node[] at ($(P1)+(-0.65,-0.4)$) {\Large $h_1$};
\node[] at ($(P2)+(0.65,-0.4)$) {\Large $h_m$};
\node[] at ($(N0)+(-0.65,-0.4)$) {\Large $h_0$};
\node[] at ($(W0)+(0.65,-0.4)$) {\Large $h_{m+1}$};
\draw [line width=0.05cm, fill=white] (W0) circle [radius=0.15]
  node[] at ($(W0)+(0,0.4)$) {\Large $\mu_m$};
\draw [line width=0.05cm, fill=white] (P1) circle [radius=0.15] node[] at ($(P1)+(0,0.4)$) {\Large $\mu_1$};

\draw [line width=0.05cm, fill=white] (P2) circle [radius=0.15] node[] at ($(P2)+(0,0.4)$) {\Large $\mu_{m-1}$};
            \end{tikzpicture}}
        \caption{The circuit $\calC_{T_4}$}
        \label{fig:circuit4}
    \end{minipage}
\end{figure}

For each type of circuit, we associate a \textit{canonical walk} $W_\calC$, defined as follows:
\begin{itemize}
  \item[($\calC_{T_1}$)] 
  For a positive circle, a canonical walk is the shortest closed walk that traverses the circle exactly once.
  For example, the walk 
  \[W_{\calC_{T_1}}:=(f_1,u_1,f_2,\ldots,u_{k-1},f_k)\]
  in the circuit $\calC_{T_1}$ depicted in Figure~\ref{fig:circuit1} is canonical.

  \item[($\calC_{T_2}$)] 
  For two negative circles joined by a path, a canonical walk is the shortest closed walk that traverses every edge of the circuit at least once; equivalently, it consists of traversing one negative circle once, then the connecting path, then the other negative circle once, and finally returning along the same path.
  For example, the walk 
  \begin{multline*}
      W_{\calC_{T_2}}:=(f_1,u_1,\ldots,u_{k-1},f_k,u_0,h_1,\mu_1,\ldots,\mu_{m-1},h_m,w_0,\\
      g_1,w_1,g_2,\ldots,w_{l-1},g_l,w_0,h_m,\mu_{m-1},\ldots,\mu_1,h_1)
  \end{multline*}
  in the circuit $\calC_{T_2}$ depicted in Figure~\ref{fig:circuit2} is canonical.

  \item[($\calC_{T_3}$)] 
  For a negative circle joined to a halfedge by a path, a canonical walk is the shortest walk that starts at the halfedge, traverses the connecting path, goes once around the negative circle, and then returns along the same path to end at the halfedge.
  For example, the walk 
  \[W_{\calC_{T_3}}:=(h_{m+1},\mu_m,\ldots,\mu_1,h_1,u_0,f_1,u_1,\ldots,u_{k-1},f_k,u_0,h_1,\mu_1,\ldots,\mu_m,h_{m+1})\]
  in the circuit $\calC_{T_3}$ depicted in Figure~\ref{fig:circuit3} is canonical.

  \item[($\calC_{T_4}$)] 
  For two halfedges joined by a path, a canonical walk is the shortest walk that starts at one halfedge and ends at the other.
  For example, the walk 
  \[W_{\calC_{T_4}}:=(h_0,\mu_0,h_1,\mu_1,\ldots,\mu_{m-1},h_m,\mu_m,h_{m+1})\]
  in the circuit $\calC_{T_4}$ depicted in Figure~\ref{fig:circuit4} is canonical.
\end{itemize}

For a walk $W:=(e_1,v_1,e_2,\dots,v_{\ell-1},e_\ell)$ on $\calH_P$, we define the map $\rho_W : E(W) \to \{\pm1\}$ as follows:
\[
\rho_W(e_i):=\begin{cases}
    1 &\text{ if }i=1, \\
    -\tau(e_{i-1},v_{i-1})\tau(e_i,v_{i-1})\rho_W(e_{i-1}) &\text{ if }i\ge 2.
\end{cases}
\]
Note that for each $i\in [\ell-1]$, we have
\begin{equation}\label{eq:1}
\begin{split}
    \tau(e_i,v_i)\rho_W(e_i)+\tau(e_{i+1},v_i)\rho(e_{i+1})&=
    \tau(e_i,v_i)\rho_W(e_i)+\tau(e_{i+1},v_i)(-\tau(e_i,v_i)\tau(e_{i+1},v_i)\rho_W(e_i))\\
    &=\tau(e_i,v_i)\rho_W(e_i)-\tau(e_i,v_i)\rho_W(e_i)=0.
\end{split}
\end{equation}
Moreover, if $W$ is a positive closed walk with the base vertex $v_\ell$, then we have
\begin{equation*}
\begin{split}
    \tau(e_\ell,v)\rho_W(e_\ell)&=(-\tau(e_\ell,v_{\ell-1})\sigma(e_\ell))\cdot(-\tau(e_{\ell-1},v_{\ell-1})\tau(e_\ell,v_{\ell-1})\rho_W(e_{\ell-1}))\\
    &=\sigma(e_\ell)\tau(e_{\ell-1},v_{\ell-1})\rho_W(e_{\ell-1}) =\sigma(e_\ell)\cdots \sigma(e_2)\tau(e_1,v_1)\rho_W(e_1) \\
    &=\sigma(e_\ell)\cdots \sigma(e_2)(-\tau(e_1,v)\sigma(e_1))\rho_W(e_1)=-\tau(e_1,v)\rho_W(e_1).
\end{split}
\end{equation*}
Thus, by combining (\ref{eq:1}), we have
\begin{align}\label{ali:1}
    \tau(e_i,v_i)\rho_W(e_i)+\tau(e_{i+1},v_i)\rho_W(e_{i+1})=0
\end{align}
for each $i\in [\ell]$, where we let $e_{\ell+1}=e_1$.

\begin{thm}[{\cite{zaslavsky1982signed}}]\label{thm:signedindependence}
 Let $P\subset \Phi_{B_d}$ be a signed poset.
 \begin{itemize}
     \item[(i)] A subset $F\subset E(\Gamma_P)$ is an independent set of $M(\widetilde{P})$ if and only if the subgraph $\calH_P(F)$ is a signed pseudo-forest.
     In particular, a subset $T\subset E(\Gamma_P)$ is a basis of $M(\widetilde{P})$ if and only if the subgraph $\Sigma(T)$ is an inclusion maximal signed pseudo-forest of $\calH_P$.
     \item[(ii)] a subset $C\subset E(\Gamma_P)$ is a circuit of $M(\widetilde{P})$ if and only if the subgraph $\Sigma(C)$ is a circuit of $\calH_P$.
 \end{itemize}
 
\end{thm}

\subsection{Toric rings of signed graphs and their divisor-theoretic properties}\label{subsection:4.2}
In this subsection, we compute the divisor class groups and weights, and characterize ($\QQ$-)Gorenstein property of toric rings of signed posets.

Let $P\subset \Phi_{B_d}$ be a signed poset.
If $\calH_P$ has connected components $\calH_{P_1},\ldots,\calH_{P_k}$, then we have $R_P\cong R_{P_1}\otimes_\kk \cdots \otimes_\kk R_{P_k}$.
Thus, in what follows, we always assume that $\calH_P$ is connected.

\begin{lem}\label{lem:a_C}
For a circuit $C\subset \widetilde{P}$ of $M(\widetilde{P})$, let $\calC:=\calH_P(C)$ be the corresponding circuit of $\calH_P$ and let $W_\calC=(e_1,v_1,e_2,\dots,v_{\ell-1},e_\ell)$ be a canonical walk on $\calC$ (if $W_\calC$ is closed, we denote its base vertex by $v_\ell$).
Then we have
    \[
    \ab_C=\sum_{e\in E(W_\calC)}\rho_{W_\calC}(e)\chi_e.
    \]
\end{lem}
\begin{proof}
    Let $\bb:=\sum_{e\in E(W_\calC)}\rho_{W_\calC}(e)\chi_e$.
    We have $\bb(e)=0$ for $e\in E(\Gamma)\setminus C$.
    Moreover, since there exists an edge $e\in E(W_\calC)$ with $\rho_{W_\calC}(e)=1$ or $-1$, the vector $\bb$ is primitive. 
    Thus, it is enough to show that $\bb\in \ker(A_{\widetilde{P}})$, that is, $(A_{\widetilde{P}}\bb)(v)=\left(\sum_{e\in E(W_\calC)}\rho_{W_\calC}(e)e\right)(v)=0$ for any $v\in V(\Gamma_P)$.
    Clearly, we have 
    \[
    \left(\sum_{e\in E(W_\calC)}\rho_{W_\calC}(e)e\right)(v)=0
    \]if $v\notin e$ for any $e\in E(W_\calC)$.
    
    Assume that $v\in e$ for some $e\in E(W_\calC)$.
    If $\calC$ is a circuit of the form $\calC_{T_3}$ or $\calC_{T_4}$, then $v\in V(W_\calC)$.
    Moreover, if $\calC$ is a circuit of the form $\calC_{T_1}$ or $\calC_{T_2}$, then $W_\calC$ is a positive closed walk.
    Therefore, it follows from (\ref{ali:1}) that
    \begin{align}\label{alig:4.3}
    \left(\sum_{e\in E(W_\calC)}\rho_{W_\calC}(e)e\right)(v)&=\sum_{\substack{e\in E(W_\calC) \\ v\in e}}\tau(e,v)\rho_{W_\calC}(e) \\
    &=\sum_{\substack{j\in [\ell]\\ v=v_j}}(\tau(e_j,v_j)\rho_{W_\calC}(e_j)+\tau(e_{j+1},v_j)\rho_{W_\calC}(e_{j+1}))=0 \notag
    \end{align}
    
\end{proof}

In what follows, we consider the following three cases separately and compute the divisor class group of $R_P$ in each case:
\begin{itemize}
    \item[Case~1;] $\calH_P$ has no signed pseudo-trees and no signed halfedges.
    \item[Case~2;] $\calH_P$ has a signed halfedge.
    \item[Case~3;] $\calH_P$ has a signed pseudo-tree but no signed halfedges.
\end{itemize}
For each of the three cases above, we fix a maximal signed pseudo-forest $\calH_P(\calF)$, where $\calF\subset E(\Gamma_P)$, as follows:
\begin{equation}\label{fixpf}
    \calH_P(\calF)\text{ is }
    \begin{cases}
        \text{a signed tree in Case~1,}\\
        \text{a signed halfedge-tree in Case~2,}\\
        \text{a signed pseudo-tree in Case~3.}
    \end{cases}
\end{equation}
In the last case, $\calH_P(\calF)$ contains a unique negative circle, so we denote its edge set by $N_\calF\subset \calF$.

Let $\epsilon_1,\ldots,\epsilon_t$ be the edges in $E(\Gamma_P)\setminus \calF$.
Note that $t=\#E(\Gamma_P)-d+1$ (resp. $t=\#E(\Gamma_P)-d$) in Case~1 (resp. in Cases~2 or 3).
For each $i=1,\ldots,t$, the subgraph $\calH_P(\calF\cup \{\epsilon_i\})$ has a unique circuit containing $\epsilon_i$.
We denote it by $\calC_i$ and call it the \textit{fundamental circuit} with respect to $\calF$.
Let $W_{\calC_i}$ be a canonical walk on $\calC_i$.

We set
\[
\calG_P:=\begin{cases}
            \ZZ^t &\text{ in Cases~1 or 2}, \\
            \ZZ^t\oplus \ZZ/2\ZZ &\text{ in Case~3}.
        \end{cases}
\]
\begin{prop}\label{prop:CL(R_P)}
Work with the above notation.
Then we have $\Cl(R_P)\cong \calG_P$.
\end{prop}
\begin{proof}
    From Theorem~\ref{thm:signedindependence} and the assumption that $\calH_P$ is connected, we have $\rank A_{\calH_P}=\#\calF=d-1$ (resp. $\rank A_{\calH_P}=\#\calF=d$) in Case~1 (resp. in Cases~2 or 3).
    It is known in \cite[Lemma~4.3]{bach2024acyclotopes} that for an independent subset $F\subset \widetilde{P}$, we have $g_{\rk(F)}(A_F)=2^{\pt(F)}$, where $\pt(F)$ denotes the number of connected signed pseudo-tree components of $\calH_P(F)$.
    In particular, since $\pt(\calF)=0$ (resp. $\pt(\calF)=1$) in Cases~1 or 2 (resp. in Case~3), we have $g_r(A_{\calH_P})=1$ (resp. $g_r(A_{\calH_P})=2$), where $r:=\rank A_{\calH_P}$.
    Moreover, we can see that $g_i(A_{\calH_P})=1$ if $i<r$ since there is a subset $F\subset \widetilde{P}$ with $\#F=i$ such that $\calH_P(F)$ is a signed tree.
    Therefore, we have $s_1=\cdots =s_{r-1}=1$ and $s_r=1$ (resp. $s_r=2$) in Cases~1 or 2 (resp. Case~3), where $s_1,\ldots,s_r$ are the invariant factors of $A_{\widetilde{P}}$.
    Thus, we get $\Cl(R_P)\cong \ZZ^{E(\Gamma_P)}/\im(A^T_{\calH_P}) \cong \calG_P$.
\end{proof}

For $e\in E(\Gamma_P)$, we define $b_e\in \ZZ^t$ and $\beta_e\in \calG_P$ as follows:
\begin{equation}\label{equ:def_weight}
b_e:=\displaystyle \sum_{i\in [t]}\sum_{\substack{f\in E(W_{\calC_i})\\ f=e}}\rho_{W_{\calC_i}}(f)\chi_i
 \quad \text{ and } \quad
\beta_e:=\begin{cases}
    b_e &\text{ in Cases~1 or 2},\vspace{0.2cm}\\
    b_e\oplus\bar{1} &\text{ in Case~3 if $e\in N_\calF$}, \vspace{0.2cm}\\ 
    b_e\oplus\bar{0} &\text{ in Case~3 if $e\notin N_\calF$. }
\end{cases}
\end{equation}
Note that the choice of the pseudo-forest $\calF$ implies that the edge $\epsilon_i$ appears exactly once in $E(W_{\calC_i})$ for each $i\in [t]$.
Moreover, $\epsilon_i$ does not appear in $E(W_{\calC_j})$ for $j\neq i$ and $\calN_\calF$.
Therefore, we have $\beta_{\epsilon_i}=\rho_{W_{\calC_i}}(\epsilon_i)\chi_i\oplus \bar{0}$.

\begin{thm}\label{thm:weightR_P}
    Work with the same notation as above.
    Let $\phi : \ZZ^{E(\Gamma_P)}\to \calG_P$ be the morphism induced by $\chi_e\mapsto \beta_e$.
    Then 
     \begin{equation}\label{cl_seq_weight}
 0 \longrightarrow \im(A^T_{\calH_P}) \lhook\joinrel\longrightarrow \ZZ^{E(\Gamma_P)} \xlongrightarrow{\phi} \calG_P \longrightarrow 0
\end{equation}
is an exact sequence, and hence $\beta_e$'s are weights of $R_P$.
\end{thm}
\begin{proof}
    We already know that $\ZZ^{E(\Gamma_P)}/\im(A^T_{\calH_P})\cong \calG_P$ from Proposition~\ref{prop:CL(R_P)}, so it is enough to show that $\im(A^T_{\calH_P})\subset \ker(\phi)$ and $\im(\phi)=\calG_P$.
    
    First, we prove that $\im(A^T_{\calH_P})\subset \ker(\phi)$.
    For a vertex $v\in V(\Gamma_P)$, we have $A^T_{\calH_P}\chi_v=\sum_{\substack{e\in E(\Gamma_P)\\ v\in e}}\tau(e,v)\chi_e$ and
    \begin{align*}
    \phi(A^T_{\calH_P}\chi_v)&=\sum_{\substack{e\in E(\Gamma_P)\\ v\in e}}\tau(e,v)\beta_e \\
    &=\rbra{\sum_{\substack{e\in E(\Gamma_P)\\ v\in e}}\tau(e,v)\rbra{\sum_{i\in [t]}\sum_{\substack{f\in E(W_{\calC_i})\\ f=e}}\rho_{W_{\calC_i}}(f)\chi_i}}\oplus\sum_{\substack{e\in N_\calF \\ v\in e}}\bar{1}.
    \end{align*}
    Note that in Cases~1 and 2, there is no $\ZZ/2\ZZ$-component.
    
    For $i\in [t]$, it follows from (\ref{alig:4.3}) that
    \begin{align*}
    \rbra{\sum_{\substack{e\in E(\Gamma_P)\\ v\in e}}\tau(e,v)\rbra{\sum_{i\in [t]}\sum_{\substack{f\in E(W_{\calC_i})\\ f=e}}\rho_{W_{\calC_i}}(f)\chi_i}}(i)&=\sum_{\substack{e\in E(\Gamma_P)\\ v\in e}}\tau(e,v)\rbra{\sum_{\substack{f\in E(W_{\calC_i})\\ f=e}}\rho_{W_{\calC_i}}(f)} \\
    &=\sum_{\substack{e\in E(W_{\calC_i}) \\ v\in e}}\tau(e,v)\rho_{W_{\calC_i}}(e)=0.
    \end{align*}
    Moreover, since $\calH_P(N_\calF)$ is a circle, the vertex $v$ is incident to either no edge in $N_\calF$ or exactly two edges in $N_\calF$.
    Thus, we have $\sum_{\substack{e\in N_\calF \\ v\in e}}\bar{1}=\bar{0}$.
    Therefore, we get $\phi(A^T_{\calH_P}\chi_v)={\bf 0}$, which implies $\im(A^T_{\calH_P})\subset \ker(\phi)$.

    Next, we show that $\im(\phi)=\calG_P$.
We know that $\beta_{\epsilon_i}=\rho_{W_{\calC_i}}(\epsilon_i)\chi_i\oplus \bar{0}$ for each $i\in [t]$.
Moreover, in Case~3, the $\ZZ/2\ZZ$-component of $\beta_{\epsilon}$ is $\bar{1}$ for every edge $\epsilon\in E_{\calF}$.
This shows that $\beta_{\epsilon_1},\ldots,\beta_{\epsilon_t}$ are generators of $\calG_P$ in Cases~1 or 2$,$ and that $\beta_{\epsilon_1},\ldots,\beta_{\epsilon_t},\beta_{\epsilon}$ generate $\calG_P$ in Case~3. Hence, $\im(\phi)=\calG_P$.
\end{proof}

\begin{ex}\label{ex:weight}
   (1) Let $P_1\subset \Phi_{B_4}$ be a signed poset with
    \begin{multline*}
    \widetilde{P_1}=\{e_1:=\chi_1+\chi_2,\; e_2:=\chi_2+\chi_3,\; e_3:=\chi_3+\chi_4,\;\\ e_4:=\chi_1+\chi_4,\; e_5:=\chi_2+\chi_4,\; e_6:=\chi_1,\; e_7:=\chi_3\},
    \end{multline*}
    see Figure~\ref{ex1}.
    \begin{figure}[H]
        \centering
        \scalebox{0.8}{
            \begin{tikzpicture}[line width=0.05cm]
\def\r{1.5};
\coordinate (N0) at ($(0,0)+({\r*cos(0)},{\r*sin(0)})$); 
\coordinate (N1) at ($(0,0)+({\r*cos(90)},{\r*sin(90)})$); 
\coordinate (N2) at ($(0,0)+({\r*cos(180)},{\r*sin(180)})$); 
\coordinate (N3) at ($(0,0)+({\r*cos(270)},{\r*sin(270)})$); 

\coordinate (P1) at ($(N0)+(1.3,0)$);
\coordinate (P2) at ($(W0)+(-1.3,0)$);

\draw[
  postaction={
    decorate,
    decoration={
      markings,
      mark=at position 0.75 with {\arrow{>}},
      mark=at position 0.25 with {\arrow{<}}
    }
  }
] (N0) -- (N1);
\draw[
  postaction={
    decorate,
    decoration={
      markings,
      mark=at position 0.75 with {\arrow{>}},
      mark=at position 0.25 with {\arrow{<}}
    }
  }
] (N1) -- (N2);
\draw[
  postaction={
    decorate,
    decoration={
      markings,
      mark=at position 0.75 with {\arrow{>}},
      mark=at position 0.25 with {\arrow{<}}
    }
  }
] (N2) -- (N3);
\draw[
  postaction={
    decorate,
    decoration={
      markings,
      mark=at position 0.75 with {\arrow{>}},
      mark=at position 0.25 with {\arrow{<}}
    }
  }
] (N3) -- (N0);
\draw[
  postaction={
    decorate,
    decoration={
      markings,
      mark=at position 0.75 with {\arrow{>}},
      mark=at position 0.25 with {\arrow{<}}
    }
  }
] (N1) -- (N3);
\draw[
  postaction={
    decorate,
    decoration={
      markings,
      mark=at position 0.5 with {\arrow{<}}
    }
  }
] (N0) -- ($(N0)+(1,0)$);
\draw[
  postaction={
    decorate,
    decoration={
      markings,
      mark=at position 0.5 with {\arrow{<}}
    }
  }
] (N2) -- ($(N2)+(-1,0)$);
\draw[dotted] ($(N0)+(1.2,0)$) -- ($(N0)+(1.5,0)$);
\draw[dotted] ($(N2)+(-1.2,0)$) -- ($(N2)+(-1.5,0)$);

\draw [line width=0.05cm, fill=white] (N0) circle [radius=0.15] node[above right] {\Large $1$};
\draw [line width=0.05cm, fill=white] (N1) circle [radius=0.15] node[] at ($(N1)+(0,0.4)$) {\Large $2$};
\draw [line width=0.05cm, fill=white] (N2) circle [radius=0.15] node[above left] {\Large $3$};
\draw [line width=0.05cm, fill=white] (N3) circle [radius=0.15] node[below left] {\Large $4$};

\node[] at ($(0,0)+({\r*cos(45)},{\r*sin(45)})$) {\Large $e_1$};
\node[] at ($(0,0)+({\r*cos(135)},{\r*sin(135)})$) {\Large $e_2$};
\node[] at ($(0,0)+({\r*cos(225)},{\r*sin(225)})$) {\Large $e_3$};
\node[] at ($(0,0)+({\r*cos(315)},{\r*sin(315)})$) {\Large $e_4$};

\node[] at (0.4,0) {\Large $e_5$};
\node[] at ($(N0)+(0.65,-0.4)$) {\Large $e_6$};
\node[] at ($(N2)+(-0.65,-0.4)$) {\Large $e_7$};
            \end{tikzpicture}}
        \caption{The Hasse diagram $\calH_{P_1}$}
        \label{ex1}
\end{figure}
    
    We can see that $\Cl(R_{P_1})\cong \ZZ^3$.
    Take the maximal signed pseudo-forest $\calF:=\{e_1,e_3,e_4,e_7\}$ and let $\epsilon_1:=e_6$, $\epsilon_2:=e_2$ and $\epsilon_3:=e_5$.
    Then 
    \begin{align*}
        W_{\calC_1}&:=(e_6,1,e_4,4,e_3,3,e_7), \\
        W_{\calC_2}&:=(e_2,2,e_1,1,e_4,4,e_3) \text{ and } \\
        W_{\calC_3}&:=(e_7,3,e_3,4,e_5,2,e_1,1,e_4,4,e_3,3,e_7).
    \end{align*}
    are canonical walks on fundamental circuits $\calC_1$, $\calC_2$ and $\calC_3$, respectively.
    Then we have 
    \begin{align*}
        \beta_{e_1}&=(0,-1,-1), &\beta_{e_2}&=(0,1,0), &\beta_{e_3}&=(1,-1,-2) , &\;&\;\\
        \beta_{e_4}&=(-1,1,1), &\beta_{e_5}&=(0,0,1), &\beta_{e_6}&=(1,0,0), &\beta_{e_7}&=(-1,0,2).
    \end{align*}

    \medskip

    (2) Let $P_2\subset \Phi_{B_4}$ be a signed poset with
    \begin{multline*}
    \widetilde{P_2}:=\{e_1:=\chi_1+\chi_2,\; e_2:=\chi_2+\chi_3,\; e_3:=\chi_3+\chi_4,\;\\ e_4:=\chi_1+\chi_4,\; e_5:=\chi_1-\chi_4,\; e_6:=\chi_3-\chi_2\},
    \end{multline*}
    see Figure~\ref{ex2}. 
\begin{figure}[H]

        \centering
        \scalebox{0.8}{
            \begin{tikzpicture}[line width=0.05cm]
\def\r{2};
\coordinate (N0) at ($(0,0)+({\r*cos(45)},{\r*sin(45)})$); 
\coordinate (N1) at ($(0,0)+({\r*cos(135)},{\r*sin(135)})$); 
\coordinate (N2) at ($(0,0)+({\r*cos(225)},{\r*sin(225)})$); 
\coordinate (N3) at ($(0,0)+({\r*cos(315)},{\r*sin(315)})$);

\draw[
  postaction={
    decorate,
    decoration={
      markings,
      mark=at position 0.75 with {\arrow{>}},
      mark=at position 0.25 with {\arrow{<}}
    }
  }
]  (N0) to [out=-30,in=30] (N3);
\draw[
  postaction={
    decorate,
    decoration={
      markings,
      mark=at position 0.75 with {\arrow{>}},
      mark=at position 0.25 with {\arrow{<}}
    }
  }
]  (N1) to [out=210,in=150] (N2);

\draw[
  postaction={
    decorate,
    decoration={
      markings,
      mark=at position 0.75 with {\arrow{>}},
      mark=at position 0.25 with {\arrow{<}}
    }
  }
] (N0) -- (N1);
\draw[
  postaction={
    decorate,
    decoration={
      markings,
      mark=at position 0.75 with {\arrow{>}},
      mark=at position 0.25 with {\arrow{>}}
    }
  }
] (N1) -- (N2);
\draw[
  postaction={
    decorate,
    decoration={
      markings,
      mark=at position 0.75 with {\arrow{>}},
      mark=at position 0.25 with {\arrow{<}}
    }
  }
] (N2) -- (N3);
\draw[
  postaction={
    decorate,
    decoration={
      markings,
      mark=at position 0.75 with {\arrow{>}},
      mark=at position 0.25 with {\arrow{>}}
    }
  }
] (N3) -- (N0);

\draw [line width=0.05cm, fill=white] (N0) circle [radius=0.15] node[above right] {\Large $1$};
\draw [line width=0.05cm, fill=white] (N1) circle [radius=0.15] node[above left]  {\Large $2$};
\draw [line width=0.05cm, fill=white] (N2) circle [radius=0.15] node[below left] {\Large $3$};
\draw [line width=0.05cm, fill=white] (N3) circle [radius=0.15] node[below right] {\Large $4$};

\node[] at ($(0,-0.3)+({\r*cos(90)},{\r*sin(90)})$) {\Large $e_1$};
\node[] at ($(-0.5,0)+({\r*cos(180)},{\r*sin(180)})$) {\Large $e_2$};
\node[] at ($(0,0.3)+({\r*cos(270)},{\r*sin(270)})$) {\Large $e_3$};
\node[] at ($(0.5,0)+({\r*cos(0)},{\r*sin(0)})$) {\Large $e_4$};

\node[] at (1,0) {\Large $e_5$};
\node[] at (-1,0) {\Large $e_6$};
            \end{tikzpicture}}
        \caption{The Hasse diagram $\calH_{P_2}$}
        \label{ex2}
\end{figure}
    We can see that $\Cl(R_{P_2})\cong \ZZ^2\oplus \ZZ/2\ZZ$.
    Take the connected signed pseudo-forest $\calF:=\{e_1,e_2,e_3,e_5\}$, which has the negative circle with the edge set $N_\calF=\calF$.
    Let $\epsilon_1:=e_4$ and $\epsilon_2:=e_6$.
    Then 
    \begin{equation}\label{equ:twowalks}
    \begin{split}
        W_{\calC_1}&:=(e_4,1,e_1,2,e_2,3,e_3), \\
        W_{\calC_2}&:=(e_6,3,e_3,4,e_5,1,e_1).
    \end{split}
    \end{equation}
    are canonical walks on fundamental circuits $\calC_1$ and $\calC_2$, respectively.
    Then we have 
    \begin{align*}
        \beta_{e_1}&=(-1,1)\oplus \bar{1}, &\beta_{e_2}&=(1,0)\oplus \bar{1}, &\beta_{e_3}&=(-1,-1)\oplus\bar{1},\\
        \beta_{e_4}&=(1,0)\oplus \bar{0}, &\beta_{e_5}&=(0,-1)\oplus \bar{1}, &\beta_{e_6}&=(0,1)\oplus \bar{0}.
    \end{align*}
   
\end{ex}

\bigskip

Using Theorem~\ref{thm:weightR_P}, we can give a criterion for $R_P$ to be ($\QQ$-)Gorenstein.

For a walk $W$ on $\calH_P$, we define the following two multisubsets of $E(W)$:
\[
W^+:=\{e\in E(W) : \rho_W(e)=1\} \quad \text{ and }\quad W^-:=\{e\in E(W) : \rho_W(e)=-1\}.
\]
\begin{thm}\label{thm:Gorenstein}
    Work with the above notation.
    Consider the following two conditions:
    \begin{itemize}
        \item[(i)] Any canonical walk $W$ on $\calH_P$ satisfies $\#W^+=\#W^-$.
        \item[(ii)] $\calH_P$ has no negative circles of odd length.
    \end{itemize}
    Then $R_P$ is $\QQ$-Gorenstein if and only if condition (i) holds.
    Moreover, $R_P$ is Gorenstein if and only if conditions (i) and (ii) hold.
\end{thm}
\begin{proof}
   Let $b_P:=\sum_{e\in E(\Gamma_P)}b_e\in \ZZ^t$ and $\beta_P:=\sum_{e\in E(\calH_P)}\beta_e\in \calG_P$.
   Since $\beta_P$ corresponds to the canonical class in $\Cl(R_P)$, we can see that $R_P$ is $\QQ$-Gorenstein (resp. Gorenstein) if and only if $\beta_P$ is a torsion element in $\calG_P$, that is, $b_P={\bf 0}$ (resp. $\beta_P={\bf 0}$).
   
   For each $i\in [t]$, we have
   \begin{equation}\label{equ:b_P}
   \begin{split}
   b_P(i)&=\rbra{\sum_{e\in E(\Gamma_P)} \sum_{i\in [t]}\sum_{\substack{f\in E(W_{\calC_i})\\ f=e}}\rho_{W_{\calC_i}}(f)\chi_i}(i) =\sum_{e\in E(\Gamma_P)}\sum_{\substack{f\in E(W_{\calC_i})\\ f=e}}\rho_{W_{\calC_i}}(f) \\
   &=\sum_{e\in E(W_{\calC_i})}\rho_{W_{\calC_i}}(e) 
   =\#W^+_{\calC_i}-\#W^-_{\calC_i}
   \end{split}
   \end{equation}
   If condition (i) is satisfied, then $b_e={\bf 0}$ by (\ref{equ:b_P}), that is, $R_P$ is $\QQ$-Gorenstein.
   
   We can see that for any circuit $\calC$ of $\calH_P$, there is a maximal signed pseudo-forest $\calF$ such that $\calC$ is a fundamental circuit with respect to $\calF$.
   Thus, assume that $R_P$ is $\QQ$-Gorenstein, that is, $b_P={\bf 0}$, then $\#W^+_\calC=\#W^-_\calC$ by (\ref{equ:b_P}).
   Therefore, $R_P$ is $\QQ$-Gorenstein if and only if condition (i) holds.

\medskip

   Suppose that conditions (i) and (ii) are satisfied.
   In Cases~1 or 2, the class group $\Cl(R_P)$ has no torsion elements.
   Thus, $\beta_P=b_P={\bf 0}$, and hence $R_P$ is Gorenstein.
   In Case~3, since the circle $\calH_P(N_\calF)$ is a negative cycle of even length, we have
   \[
   \beta_P=b_P\oplus \sum_{e\in N_\calF}\bar{1}=b_P\oplus \bar{0}={\bf 0}.
   \]
   Therefore, conditions (i) and (ii) imply that $R_P$ is Gorenstein.

   Assume that $R_P$ is Gorenstein.
   Then condition (i) holds since $R_P$ is $\QQ$-Gorenstein.
   
   In Case~1, $\calH_P$ has no negative circles of odd length because it has no signed pseudo-trees.
   
   Consider Case~3. If $\calH_P$ has a negative circle $\calN$ of odd length, then $\calH_P$ has a maximal signed pseudo-tree $\calF$ containing $\calN$.
   Thus, the sum of weights induced by $\calF$ is nonzero, contradicting the assumption that $R_P$ is Gorenstein.
   
   In Case~2, suppose that $\calH_P$ has a negative circle $\calN$ of odd length.
   Then $\calH_P$ has a fundamental circuit $\calC_i$ of the form $\calC_{T_3}$ containing $\calN$.
   Take the canonical walk
   \[
   W_{\calC_i}:=(h_{m+1},\mu_m,\ldots,\mu_1,h_1,u_0,f_1,u_1,\ldots,u_{k-1},f_k,u_0,h_1,\mu_1,\ldots,\mu_m,h_{m+1})
   \]
   as in Figure~\ref{fig:circuit3}.
   Then we have
   \[
   b_P(i)=\sum_{e\in E(W_{\calC_i})}\rho_{W_{\calC_i}}(e)=2\sum_{j=1}^{m+1}\rho_{W_{\calC_i}}(h_j)+\sum_{l=1}^k \rho_{W_{\calC_i}}(f_l).
   \]
   Since $\calC$ has an odd length, $\sum_{l=1}^k \rho_{W_{\calC_i}}(f_l)$ is odd, and hence $b_P(i)\neq 0$.
   This contradicts the fact that condition (i) holds.
   Therefore, we conclude that $R_P$ is Gorenstein if and only if conditions (i) and (ii) hold.
\end{proof}

Let $G$ be a simple graph on vertex set $V(G)=[d]$ with edge set $E(G)$.
Then we obtain the signed poset
\[
P_G:=\{\chi_v+\chi_w : \{v,w\}\in E(G)\}\subset \Phi_{B_d}.
\]
It is easy to see that $\widetilde{P_G}=P_G$.
Moreover, the Hasse diagram of $P_G$ is the bidirected graph on the underlying graph $G$ obtained by setting $\tau(e,v)=+$ for every incidence $(e,v)$.
Equivalently, each edge is directed into both of its endpoints, so this is an orientation of the all-negative signed graph on $G$.

We can characterize when $R_{P_G}$ is ($\QQ$-)Gorenstein using Theorem~\ref{thm:Gorenstein} as follows:
\begin{cor}
    Let $G$ be a simple graph.
    Then $R_{P_G}$ is $\QQ$-Gorenstein.
    Moreover, $R_{P_G}$ is Gorenstein if and only if $G$ is bipartite.
\end{cor}
\begin{proof}
    Since $\calH_{P_G}$ has no halfedges, all circuits in $\calH_{P_G}$ are of the forms $\calC_{T_1}$ or $\calC_{T_2}$.
    Moreover, any canonical walk
    \[
    W:=(e_1,v_1,e_2,\dots,v_{\ell-1},e_\ell)
    \]
    has even length and satisfies $\rho_W(e_i)+\rho_W(e_{i+1})=0$ for each $i\in [\ell]$ by its orientation, so we have $\#W^+=\#W^-$.
    Therefore, $R_{P_G}$ is $\QQ$-Gorenstein from Theorem~\ref{thm:Gorenstein}.

    Furthermore, we can see that a circle in $\calH_{P_G}$ is negative if and only if it has an odd length.
    This shows that $R_{P_G}$ is Gorenstein if and only if $G$ has no odd cycles, that is, $G$ is bipartite.
\end{proof}

\subsection{Conic divisorial ideals of toric rings of signed posets}
In this subsection, we describe a region representing conic classes in the divisor class group of the toric ring of a signed poset.
Throughout this subsection, let $P\subset \Phi_{B_d}$ be a singed poset such that $\calH_P$ is connected.
In addition, let $\calF$ be a subset of $E(\Gamma_P)$ such that $\calH_P(\calF)$ is a maximal signed pseudo-forest as in (\ref{fixpf}).
Moreover, let $\epsilon_1,\ldots,\epsilon_t$ be the edges in $E(\Gamma_P)\setminus \calF$ and let $\calC_i$ be the fundamental circuit with respect to $\calF$ for each $i\in [t]$.
Furthermore, for $e\in E(\Gamma_P)$, let $\beta_e\in \calG_P$ be the weight defined as (\ref{equ:def_weight}), and let $\calB:=\{\widetilde{\beta_e}=b_e\in \ZZ^t : e\in E(\Gamma_P)\}$.
\begin{thm}\label{thm:weightconic}
    We have 
    \[
    \calW(\calB)=\bigcap_{\calC\in \fkC(\calH_P)}\set{\zb \in \RR^t : -\#W_\calC^-< \sum_{\epsilon_i\in W_{\calC}^+}\rho_{W_{\calC_i}}(\epsilon_i)\zb(i)-\sum_{\epsilon_j\in W_{\calC}^-}\rho_{W_{\calC_j}}(\epsilon_j)\zb(j)< \#W^+_\calC}.
    \]
\end{thm}
\begin{proof}
    Note that $\Cl(R_P)_{\mathrm{fre}}=\ZZ^t$ and $\Hom_\ZZ(\ZZ^t,\ZZ)\cong \ker(A_{\widetilde{P}})$.
    For a circuit $\calC\in \fkC(\calH_P)$, let $C\subset E(\Gamma_P)=\widetilde{P}$ be the edge set of $\calC$.
    Then $C$ is a circuit of $M(\widetilde{P})$ and $\ab_C=\sum_{e\in E(W_\calC)}\rho_{W_\calC}(e)\chi_e$ by Lemma~\ref{lem:a_C}.
    In particular, we have $\ab_C^*(b_e)=\ab_C(e)$ for $e\in E(\Gamma_P)$.
    
    Since $b_{\epsilon_i}=\rho_{W_{\calC_i}}(\epsilon_i)\chi_i$, for $\zb\in \ZZ^t$, we can write $\zb=\sum_{i=1}^t\zb(i)\chi_i=\sum_{i=1}^t\zb(i)\rho_{W_{\calC_i}}(\epsilon_i)b_{\epsilon_i}$ and have
    \begin{align*}
        \ab^*_C(\zb)&=\sum_{i=1}^t\rho_{W_{\calC_i}}(\epsilon_i)\ab^*_C(b_{\epsilon_i})\zb(i)=\sum_{i=1}^t\rho_{W_{\calC_i}}(\epsilon_i)\ab_C(\epsilon_i)\zb(i) \\
        &=\sum_{i=1}^t\rho_{W_{\calC_i}}(\epsilon_i)\rbra{\sum_{\substack{e\in E(W_\calC) \\ e=\epsilon_i}}\rho_{W_\calC}(e)}\zb(i)=\sum_{i=1}^t\rho_{W_{\calC_i}}(\epsilon_i)\rbra{\sum_{\substack{e\in W^+_\calC \\ e=\epsilon_i}}\zb(i) - \sum_{\substack{e\in W^-_\calC \\ e=\epsilon_i}}\zb(i)} \\
        &=\sum_{\epsilon_i\in W_{\calC}^+}\rho_{W_{\calC_i}}(\epsilon_i)\zb(i)-\sum_{\epsilon_j\in W_{\calC}^-}\rho_{W_{\calC_j}}(\epsilon_j)\zb(j).
    \end{align*}
    We can see that $|\ab^+_C|=\# W^+_{\calC}$ and $|\ab^-_C|=\# W^-_{\calC}$.
    Therefore, from Theorem~\ref{thm:conicdesc}, we get the desired result.
\end{proof}

\begin{ex}
    We use the poset $P_2$ as in Example~\ref{ex:weight} (2).
    Then $\calH_{P_2}$ has four circuits; $\calC_1$ and $\calC_2$ with their canonical walks defined in (\ref{equ:twowalks}), and two exceptional circuits $\calC_3$ and $\calC_4$ with their canonical walks
    \begin{align*}
        W_{\calC_3}&:=(e_2,3,e_6,2,e_1,1,e_5,4,e_4,1,e_1), \\
        W_{\calC_4}&:=(e_2,2,e_6,3,e_3,4,e_5,1,e_4,4,e_3)
    \end{align*}
    We have
    \begin{align*}
        W_{\calC_1}^+&=\{e_4,e_2\}, &W_{\calC_2}^+&=\{e_6,e_3\}, &W_{\calC_3}^+&=\{e_2,e_5,e_4\}, &W_{\calC_4}^+&=\{e_2,e_6,e_4\} \\
        W_{\calC_1}^{-}&=\{e_1,e_3\}, &W_{\calC_2}^{-}&=\{e_5,e_1\}, &W_{\calC_3}^-&=\{e_6,e_1,e_1\}, &W_{\calC_4}^-&=\{e_3,e_5,e_3\}
    \end{align*}
    Thus, we get
    \[
    \calW(\calB)=\set{(z_1,z_2) \in \RR^2 : \hspace{-0.1cm}
\begin{array}{cccc}
-2 < z_1 < 2, \vspace{0.1cm}\\
 -2 < z_2 < 2, \vspace{0.1cm}\\
-3 < z_1-z_2 < 3, \vspace{0.1cm}\\
-3<z_1+z_2<3
\end{array}}.
    \]
\end{ex}

\section{Toric rings of balanced signed posets}\label{sec:balanced}

In this section, we focus on the toric rings of balanced signed posets.

\subsection{Balanced signed posets}
First, we recall balanced signed posets.
Let $\Sigma=(\Gamma,\sigma)$ be a signed graph with orientation $\tau$.
Recall that $\Sigma$ has no loops and no loose edges as in Section~\ref{sec:4}.

We say that $\Sigma$ is \textit{balanced} if any circle on $\Sigma$ is positive.
It is known that $\Sigma$ is balanced if and only if $\Sigma$ can be converted into a signed graph with only positive edges except for halfedges by \textit{switching} operations, that is, for a fixed vertex $v$ flipping the sign of $\tau(v,e)$ for all node-edge-incidences involving $v$ (\cite[Corollary~3.3]{zaslavsky1982signed}).
Moreover, $\Sigma$ is balanced if and only if $A_{\Sigma}$ is \textit{totally unimodular}, that is, every minor of $A_{\Sigma}$ is $0$ or $\pm 1$ (\cite[Theorem~8A.5]{zaslavsky1982signed}).

We say that a signed poset $P\subset \Phi_{B_d}$ is \textit{balanced} if its Hasse diagram $\calH_P$ is balanced.
We now explain that the study of the toric ring of a balanced signed poset reduces to that of a signed poset in a root system of type $A$.

Let $P\subset \Phi_{B_d}$ be a balanced signed poset whose Hasse diagram is connected.
We construct from $\calH_P$ a new signed poset in $\Phi_{A_d}$ as follows.
Since $\calH_P$ is balanced, we can obtain a signed graph $\Sigma$ such that every ordinary edge of $\Sigma$ is positive by applying a suitable sequence of switchings to $\calH_P$.

If $\Sigma$ has a halfedge, adjoin a new vertex $d+1$ to the underlying graph of $\Sigma$. For each halfedge $e$, replace it with an ordinary edge $e'$ joining its original endpoint $v$ to $d+1$, and choose the local direction at $v$ to be the same as that of the halfedge $e$. The local direction at $d+1$ is then determined so that the resulting edge is positive. If $\Sigma$ has no halfedges, we do nothing.

Let $P^b$ be the signed poset whose Hasse diagram coincides with this new oriented signed graph.
The Hasse diagram $\calH_{P^b}$ has no halfedges and its all edges are positive, and hence $P^b\subset \Phi_{A_d}$.
Moreover, its bidirection defines an ordinary orientation, and therefore $\calH_{P^b}$ may be viewed as a directed graph.
We denote this orientation of $\Gamma_{P^b}$ by $\scrO_{P^b}$.

\begin{ex}\label{ex:hibi}
    Let $P\subset \Phi_{B_4}$ be a signed poset with
    \[
    \widetilde{P}:=\{\chi_3+\chi_1,\; \chi_4+\chi_1,\; \chi_3+\chi_2,\; \chi_4+\chi_2,\; \chi_1,\; \chi_2\}.
    \]
    Then we can see that $P$ is balanced.
    Applying the switching operation to the vertices $1$ and $2$, and adding the new vertex $5$, we obtain the signed poset $P^b$ with
     \[
        \widetilde{P^b}:=\{\chi_3-\chi_1,\; \chi_4-\chi_1,\; \chi_3-\chi_2,\; \chi_4-\chi_2,\; \chi_5-\chi_1,\; \chi_5-\chi_2\}.
    \]

    \begin{figure}[H]
    \centering
    \begin{minipage}{0.46\linewidth}
        \centering
        \scalebox{0.9}{
            \begin{tikzpicture}[line width=0.05cm]
            \coordinate (a1) at (0, 0); 
      \coordinate (a2) at (2, 0); 
      \coordinate (a3) at (0, 2); 
      \coordinate (a4) at (2, 2);

      \draw[
  postaction={
    decorate,
    decoration={
      markings,
      mark=at position 0.75 with {\arrow{>}},
      mark=at position 0.25 with {\arrow{<}}
    }
  }
] (a1) -- (a3);
\draw[
  postaction={
    decorate,
    decoration={
      markings,
      mark=at position 0.75 with {\arrow{>}},
      mark=at position 0.25 with {\arrow{<}}
    }
  }
] (a1) -- (a4);
\draw[
  postaction={
    decorate,
    decoration={
      markings,
      mark=at position 0.75 with {\arrow{>}},
      mark=at position 0.25 with {\arrow{<}}
    }
  }
] (a2) -- (a3);
\draw[
  postaction={
    decorate,
    decoration={
      markings,
      mark=at position 0.75 with {\arrow{>}},
      mark=at position 0.25 with {\arrow{<}}
    }
  }
] (a2) -- (a4);
\draw[
  postaction={
    decorate,
    decoration={
      markings,
      mark=at position 0.65 with {\arrow{<}}
    }
  }
] (a1) -- (0,-0.8);
\draw[
  postaction={
    decorate,
    decoration={
      markings,
      mark=at position 0.65 with {\arrow{<}}
    }
  }
] (a2) -- (2,-0.8);
\draw[dotted] (0,-1)--(0,-1.3);
\draw[dotted] (2,-1)--(2,-1.3);
      \draw [line width=0.05cm, fill=white] (a1) circle [radius=0.15] node [left=5pt] {\large $1$};
      \draw [line width=0.05cm, fill=white] (a2) circle [radius=0.15] node [right=5pt] {\large $2$};
      \draw [line width=0.05cm, fill=white] (a3) circle [radius=0.15] node [left=5pt] {\large $3$};
      \draw [line width=0.05cm, fill=white] (a4) circle [radius=0.15] node [right=5pt] {\large $4$};
            \end{tikzpicture}}
        \caption{The Hasse diagram $\calH_P$}
        \label{fig:beforereduced}
    \end{minipage}
    \begin{minipage}{0.46\linewidth}
        \centering
        \scalebox{0.9}{
        \begin{tikzpicture}[line width=0.05cm]
        \coordinate (a1) at (0, 0); 
      \coordinate (a2) at (2, 0); 
      \coordinate (a3) at (-1, 2); 
      \coordinate (a4) at (1, 2);
      \coordinate (a5) at (3,2);

      \draw[
  postaction={
    decorate,
    decoration={
      markings,
      mark=at position 0.5 with {\arrow{>}}
    }
  }
] (a1) -- (a3);
\draw[
  postaction={
    decorate,
    decoration={
      markings,
      mark=at position 0.65 with {\arrow{>}}
    }
  }
] (a1) -- (a4);
\draw[
  postaction={
    decorate,
    decoration={
      markings,
      mark=at position 0.65 with {\arrow{>}}
    }
  }
] (a2) -- (a3);
\draw[
  postaction={
    decorate,
    decoration={
      markings,
      mark=at position 0.65 with {\arrow{>}}
    }
  }
] (a2) -- (a4);
\draw[
  postaction={
    decorate,
    decoration={
      markings,
      mark=at position 0.65 with {\arrow{>}}
    }
  }
] (a1) -- (a5);
\draw[
  postaction={
    decorate,
    decoration={
      markings,
      mark=at position 0.5 with {\arrow{>}}
    }
  }
] (a2) -- (a5);
      \draw [line width=0.05cm, fill=white] (a1) circle [radius=0.15] node [left=5pt] {\large $1$};
      \draw [line width=0.05cm, fill=white] (a2) circle [radius=0.15] node [right=5pt] {\large $2$};
      \draw [line width=0.05cm, fill=white] (a3) circle [radius=0.15] node [left=5pt] {\large $3$};
      \draw [line width=0.05cm, fill=white] (a4) circle [radius=0.15] node [right=5pt] {\large $4$};
      \draw [line width=0.05cm, fill=white] (a5) circle [radius=0.15] node [right=5pt] {\large $5$};
            \end{tikzpicture}}
        \caption{The graph $\Gamma_{P^b}$ with orientation $\scrO_{P^b}$}
        \label{fig:afterreduced}
    \end{minipage}
\end{figure}
\end{ex}

The passage from $P$ to $P^b$ has a simple ring-theoretic interpretation. Switching corresponds to changing the signs of the corresponding coordinates in $\RR^d$, while adjoining the new vertex $d+1$ corresponds to a Laurent polynomial extension. Hence one obtains
\[
R_P[t_{d+1}^{\pm 1}] \cong R_{P^b}.
\]
In particular, divisor-theoretic properties such as the divisor class group and the Gorenstein property of $R_P$ are reduced to those of $R_{P^b}$.

\medskip

Let $\Pi:=\{p_1,\ldots,p_{d-1}\}$ be a poset equipped with a partial order $\preceq$.
For a subset $I \subset \Pi$, we say that $I$ is a \textit{filter} or \textit{upper set} of $\Pi$ if $p \in I$ and $p \preceq q$ imply $q \in I$. 

We define a $\kk$-algebra associated with $\Pi$ as follows:
\[
\kk[\Pi]:=\kk\left[\; \left(\prod_{p_i\in I}t_i\right) t_{d} : \text{$I$ is a filter of $P$}\right]\subset \kk[t_1,\ldots,t_{d}].
\]
This $\kk$-algebra is called the {\em Hibi ring} of $P$.
Hibi rings are usually defined via poset ideals (down sets).
In this paper, however, we use filters in place of poset ideals, in order to make our notation compatible with the orientation conventions for Hasse diagrams.

Set $\widehat{\Pi}=\Pi \cup \{p_0, p_d\}$, 
where $p_0$ (resp. $p_d$) is a new minimum (resp. maximum) element not belonging to $\Pi$.
Let $H_{\widehat{\Pi}}$ be the ordinary Hasse diagram of $\widehat{\Pi}$.
For each edge $e:=\{p_i\prec p_j\}$ of $H_{\widehat{\Pi}}$, let
\begin{align*}
\vb_e:=
\begin{cases}
\chi_j-\chi_i \;&\text{ if }i \not= 0, \\
\chi_j \; &\text{ if }i=0 
\end{cases}\; \in \RR^{d}.
\end{align*}
Moreover, let $E(\widehat{\Pi}):=\{\vb_e : e\text{ is an edge of }H_{\widehat{\Pi}}\}\subset \Phi_{B_d}$ and let $P_\Pi\subset \Phi_{B_d}$ be the positive linear closure of $E(\widehat{\Pi})$.

Then we can see that $P_{\Pi}$ is a balanced signed poset, $\widetilde{P_{\Pi}}=E(\widehat{\Pi})$ and that $\Gamma_{P^b_\Pi}$ with orientation $\scrO_{P^b_\Pi}$ coincides with $H_{\widehat{\Pi}}$.
Moreover, we have $R_{P^b_{\Pi}}\cong R_{P_\Pi}[t_0^{\pm 1}]=\kk[\Pi][t_0^{\pm 1}]$, which shows that any Hibi ring is isomorphic to the toric ring of a balanced signed poset.

On the other hand, there is the toric ring of a balanced signed poset which is not isomorphic to any Hibi ring as follows:
\begin{ex}
    Let $P\subset \Phi_{B_4}$ be the signed poset defined in Example~\ref{ex:hibi}.

    {\tt MAGMA}(\cite{bosma1997magma}) tells us that $\{(1,0,0,1),(0,1,0,0),(0,0,1,0),(0,0,0,1),(1,1,-1,-1)\}$ is the Hilbert basis of the cone $C^{\vee}(\widetilde{P})$.
    Therefore, we have
    \[
    R_P=\kk[t_1,t_2,t_3,t_4,t_1t_2t_3^{-1}t_4^{-1}]\cong \kk[x_1,x_2,x_3,x_4,x_5]/(x_1x_2-x_3x_4x_5).
    \]
    Moreover, we have $\dim R_P=4$ and $\Cl(R_P)\cong \ZZ^2$ from Proposition~\ref{prop:CL(R_P)}.

    If there exists a poset $\Pi$ with $\kk[\Pi]\cong R_P$, then $\Pi$ must satisfy $\#\Pi=\dim R_{P_\Pi} -1=3$ and $\#P_\Pi=\dim R_{P_\Pi}+\rank \Cl(R_{P_\Pi})=6$.
    In this case, we can see that $\Pi$ is 3-element antichain, that is, $p_i$ and $p_j$ are incomparable for any two distinct elements $p_i,p_j\in \Pi$, and that $\kk[\Pi]$ is not hypersurface while $R_P$ is.
    Thus, there is no poset $\Pi$ with $\kk[\Pi]\cong R_P$.
\end{ex}

\subsection{Conic divisorial ideals of toric rings of balanced posets}
In this subsection, we characterize the conic divisorial ideals of the toric ring of a balanced signed poset.
In the previous subsection, we showed that the case of balanced signed posets can be reduced to type $A$. 
Therefore, from now on, we always consider signed posets in $\Phi_{A_{d-1}}$.

Throughout this subsection, let $P \subset \Phi_{A_{d-1}}$ be a signed poset whose Hasse diagram is connected.
As mentioned in the previous section, $\calH_P$ is regarded as a directed graph with underlying graph $\Gamma_P$ and orientation $\scrO_P$.

Let $\calF$ be a subset of $E(\Gamma_P)$ such that $\calH_P(\calF)$ is a signed tree.
Moreover, let $\epsilon_1,\ldots,\epsilon_t$ be the edges in $E(\Gamma_P)\setminus \calF$ and let $\calC_i$ be the fundamental circuit with respect to $\calF$ for each $i\in [t]$.
Furthermore, let $\calB:=\{\beta_e : e\in E(\Gamma_P) \}$ be the set of weights of $R_P$ defined as in Section~\ref{subsection:4.2}, and let $\phi : \ZZ^{E(\Gamma_P)}\to \calG_P=\ZZ^t$ be the morphism induced by $\chi_e\mapsto \beta_e$.

\medskip

Since $\calH_P$ has no negative circles and no halfedges, a circuit of $M(\widetilde{P})$ one-to-one corresponds to a cycle of $\Gamma_P$.
For a cycle $\calC:=(i_0,i_1,\ldots, i_{\ell-1},i_\ell=i_0)$ of $\Gamma_P$, where $i_j\in V(\Gamma_P)$ and $\{i_{j-1},i_j\}\in E(\Gamma_P)$, we set
\[
\calC^{\uparrow}:=\left\{\{i_{j-1},i_j\} : \chi_{i_j}-\chi_{i_{j-1}}\in \widetilde{P}\right\}, \quad \calC^{\downarrow}:=\left\{\{i_{j-1},i_j\} : \chi_{i_{j-1}}-\chi_{i_j}\in \widetilde{P}\right\}.
\]
Note that we can take a canonical walk $W_\calC$ on $\calC$ with $W^+_{\calC}=\calC^{\uparrow}$ and $W^-_{\calC}=\calC^{\downarrow}$.
In particular, we may assume that $\epsilon_i\in \calC^{\uparrow}_i$ for each $i\in [t]$.

\medskip

From Theorem~\ref{thm:weightconic}, we have the following theorem, which generalizes the result of Hibi rings (\cite[Theorem~1.3]{higashitani2019conic}, \cite[Theorem~1.2]{matsushita2022conic}):
\begin{thm}\label{thm:weightAconic}
Work with the above notation and assumptions.
Then we have
    \[
\calW(\calB)=\bigcap_{\calC}\left\{\zb \in \RR^t : -\#\calC^\downarrow<\sum_{\epsilon_i\in \calC^\uparrow}\zb(i)- \sum_{\epsilon_j\in \calC^\downarrow}\zb(j)< \#\calC^\uparrow\right\},
\]
where $\calC$ runs over all cycles in $\Gamma_P$.
\end{thm}

Moreover, the conic divisorial ideals of the toric ring of a balanced signed poset can be characterized in terms of certain acyclic orientations of a graph.
\begin{thm}\label{thm:characonicA}
    There is a bijection among the following three sets:
    \begin{itemize}
        \item[(i)] the isomorphism classes of conic divisorial ideals of $R_P$;
        \item[(ii)] the lattice points in $\calW(\calB)$;
        \item[(iii)] the acyclic orientations of $\Gamma_{P}$ with a chosen vertex as the unique source.
    \end{itemize}
\end{thm}
This theorem generalizes the result of \cite[Theorem~2.3.13]{robinson2020big} for Hibi rings.
Although the proof follows a similar idea, we give a simpler and shorter argument.

We need some notation and lemmas.
Let $\scrA(\Gamma_{P})$ be the set of acyclic orientations of $\Gamma_{P}$.
Note that $\scrO_P\in \scrA(\Gamma_P)$.
In addition, for a vertex $v\in V(\Gamma_{P})$, let $\scrA(\Gamma_{P},v)$ be the set of acyclic orientations of $\Gamma_{P}$ with a unique source at vertex $v$.

Since $P$ is a balanced signed poset, $A_{\widetilde{P}}$ is totally unimodular, and hence the multiplicity Tutte polynomial $M_{\widetilde{P}}(x,y)$ coincides with the Tutte polynomial of $\Gamma_P$.
Here, the \textit{Tutte polynomial} of a simple graph $G=(V,E)$ is defined by
\[
T_G(x,y):=\sum_{A\subset E}(x-1)^{c(A)-c(E)}(y-1)^{c(A)+\#A-\#V},
\]
where $c(A)$ denotes the number of connected components of the subgraph induced by $A$.
It is known that $T_{\Gamma_P}(1,0)=\# \scrA(\Gamma_P,v)$ for any $v\in V(\Gamma_P)$ (\cite[Theorem~7.3]{greene1983onthe}).

For two orientations $\scrO_1$ and $\scrO_2$ of $\Gamma_P$, we define $E(\scrO_1,\scrO_2)$ to be the set of edges of $\Gamma_{P}$ that are oriented in the same direction in both $\scrO_1$ and $\scrO_2$.

\begin{lem}\label{lem:acyclic}
    For any $\scrO\in \scrA(\Gamma_{P})$, we have $\phi(\chi_{E(\scrO,\scrO_{P})})\in \calW(B)$.
\end{lem}
\begin{proof}
    Assume that $\bb:=\phi(\chi_{E(\scrO,\scrO_{P})})=\sum_{e\in E(\scrO,\scrO_{P})}\beta_e\notin \calW(\calB)$.

    By Theorem~\ref{thm:weightAconic}, there exists a cycle $\calC$ of $\Gamma_P$ with $\ab^*_C(\bb)\ge \#\calC^{\uparrow}$ or $\ab^*_C(\bb)\le -\#\calC^{\downarrow}$, where $C\subset E(\Gamma_P)$ is the circuit of $M(\widetilde{P})$ corresponding to $\calC$.

    On the other hand, for a cycle $\calC$ of $\Gamma_P$, it follows from Lemma~\ref{lem:a_C} that
    \[
    \ab^*_C(\bb)=\sum_{e\in E(\scrO,\scrO_P)}\ab^*_C(\beta_e)=\sum_{e\in E(\scrO,\scrO_P)}\ab_C(e)=\#(E(\scrO,\scrO_P)\cap \calC^{\uparrow})-\#(E(\scrO,\scrO_P)\cap \calC^\downarrow).
    \]
    
    Therefore, if $\ab^*_C(\bb)\ge \#\calC^{\uparrow}$ (resp. $\ab^*_C(\bb)\le -\#\calC^{\downarrow}$), then $E(\scrO,\scrO_P)\subset \calC^{\uparrow}$ and $E(\scrO,\scrO_P)\cap \calC^{\downarrow}=\emptyset$ (resp. $E(\scrO,\scrO_P)\subset \calC^{\downarrow}$ and $E(\scrO,\scrO_P)\cap \calC^{\uparrow}=\emptyset$).
    This implies that $\calC$ is a directed cycle in $\scrO$, contradicting $\scrO\in \scrA(\Gamma_P)$.
\end{proof}

Let $d_{\Gamma_P}\in \ZZ^{V(\Gamma_P)}$ denote the vector whose $v$-th component is the degree of $v$ in $\Gamma_P$.
Similarly, for an orientation $\scrO$ of $\Gamma_P$, let $d^{\ini}_{\scrO}\in \ZZ^{V(\Gamma_P)}$ (resp. $d^{\out}_\scrO\in \ZZ^{V(\Gamma_P)}$) denote the vector whose $v$-th component is the in-degree (resp. out-degree) of $v$ in $\scrO$.
\begin{lem}\label{lem:chip}
    Let $\scrO_1$ and $\scrO_2$ be two orientations of $\Gamma_{P}$.
    Then we have
    \[
    A_{\widetilde{P}}(\chi_{E(\scrO_1,\scrO_P)}-\chi_{E(\scrO_2,\scrO_P)})=d^{\out}_{\scrO_2}-d^{\out}_{\scrO_1}.
    \]
\end{lem}
\begin{proof}
    For each $i=1,2$, let $\psi_i:=2\chi_{E(\scrO_i,\scrO_P)}-\chi_{E(\Gamma_P)}$.
    Then we can see that 
    \[
    A_{\widetilde{P}}\psi_i=d^{\ini}_{\scrO_i}-d^{\out}_{\scrO_i}=d_{\Gamma_P}-2d^{\out}_{\scrO_i}.\]
    Thus, we have 
    \[
    2A_{\widetilde{P}}(\chi_{E(\scrO_1,\scrO_P)}-\chi_{E(\scrO_2,\scrO_P)})=A_{\widetilde{P}}(\psi_1-\psi_2)=2(d^{\out}_{\scrO_2}-d^{\out}_{\scrO_1}),
    \]
    and hence get the desired result.
\end{proof}

\begin{proof}[{proof of Theorem~\ref{thm:characonicA}}]
    We already know that there is a bijection between (i) and (ii) (Proposition~\ref{prop:coniccorres}).
    We fix a vertex $v$ of $\Gamma_P$ and show that the map $f: \scrA(\Gamma_P,v) \to \calW(\calB)\cap \ZZ^t$ ($\scrO\mapsto \phi(\chi_{E(\scrO,\scrO_P)})$) is bijective.
    Note that this map is well-defined by Lemma~\ref{lem:acyclic}.

    Moreover, from Corollary~\ref{cor:numberconic} and Remark~\ref{rem:number}, we have
    \[
    \#(\calW(\calB)\cap \ZZ^t)=M_{\widetilde{P}}(1,0)=T_{\Gamma_P}(1,0)=\#\scrA(\Gamma_P,v)
    \]
    Thus, it is enough to show that $f$ is injective.
    Assume that there are two orientations $\scrO_1,\scrO_2\in \scrA(\Gamma_P,v)$ with $\phi(\chi_{E(\scrO_1,\scrO_P)})=\phi(\chi_{E(\scrO_2,\scrO_P)})$.
    Then it follows from (\ref{cl_seq_weight}) that $\chi_{E(\scrO_1,\scrO_P)}-\chi_{E(\scrO_2,\scrO_P)}\in \im(A_{\widetilde{P}}^T)$ there exists $\xb\in \ZZ^{V(\Gamma_P)}$ with $\chi_{E(\scrO_1,\scrO_P)}-\chi_{E(\scrO_2,\scrO_P)}=A_{\widetilde{P}}^T\xb$.
    Thus, by Lemma~\ref{lem:chip}, we have
    \begin{align*}
        d^{\out}_{\scrO_2}-d^{\out}_{\scrO_1}=A_{\widetilde{P}}(\chi_{E(\scrO_1,\scrO_P)}-\chi_{E(\scrO_2,\scrO_P)})=A_{\widetilde{P}}A^T_{\widetilde{P}}\xb.
    \end{align*}
    In this case, according to theory of chip-firing game (\cite[Lemma~14.13.2]{godsil2001algebraic}), we have $d^{\out}_{\scrO_2}=d^{\out}_{\scrO_1}$, and hence $\chi_{E(\scrO_1,\scrO_P)}-\chi_{E(\scrO_2,\scrO_P)}\in \ker(A_{\widetilde{P}})$.

    Therefore, we have $\chi_{E(\scrO_1,\scrO_P)}-\chi_{E(\scrO_2,\scrO_P)}\in \ker(A_{\widetilde{P}})\cap \im(A_{\widetilde{P}}^T)=\{\bf 0\}$
    Thus, $\chi_{E(\scrO_1,\scrO_P)}=\chi_{E(\scrO_2,\scrO_P)}$, which shows that $f$ is injective.
\end{proof}

\begin{ex}
Let $P\subset \Phi_{A_4}$ be the signed poset with
    \begin{multline*}
        \widetilde{P}:=\{e_1:=\chi_3-\chi_1,\; e_2:=\chi_4-\chi_1,\; e_3:=\chi_3-\chi_2, \\
        \; e_4:=\chi_4-\chi_2,\; e_5:=\chi_5-\chi_1,\; e_6:=\chi_5-\chi_2\},
    \end{multline*}
    which coincides with $P^b$ defined in Example~\ref{ex:hibi}.
    
Let $\calF:=\{e_2,e_3,e_4,e_6\}\subset \widetilde{P}$, then $\calH_P(\calF)$ is a tree.
Let $\epsilon_1:=e_1$ and $\epsilon_2:=e_5$.
Then we obtain the fundamental cycles
\[
\calC_1=(1,3,2,4,1) \quad \text{ and } \quad \calC_2=(1,5,2,4,1)
\]
and we can compute the weight vectors using Theorem~\ref{thm:weightR_P} as follows:
\begin{align*}
        \beta_{e_1}&=(1,0), &\beta_{e_2}&=(-1,-1), &\beta_{e_3}&=(-1,0), \\
        \beta_{e_4}&=(1,1), &\beta_{e_5}&=(0,1), &\beta_{e_6}&=(0,-1).
\end{align*}
From Theorem~\ref{thm:weightAconic}, we have
\[
\calW(\calB)\cap \ZZ^2=\left\{(z_1,z_2) \in \ZZ^2 \; : \; \begin{array}{ccc}
-2 < z_1 < 2, \\ 
-2 < z_2 < 2, \\
-2< z_1-z_2 <2
\end{array}
\right\}=\set{\begin{array}{ccc}
(0,1),\; (1,1), \\ 
(-1,0),\; (0,0),\; (1,0), \\
(-1,-1),\; (0,-1)
\end{array}}.
\]

The graph $\Gamma_P$ has 7 acyclic orientations with unique source vertex $1\in V(\Gamma_P)$ (see Figure~\ref{Figure:Norientation}), which one-to-one correspond to the lattice points in $\calW(\calB)$ as follows:
\begin{itemize}
    \item[(A)] $\beta_{e_1}+\beta_{e_2}+\beta_{e_5}+\beta_{e_3}+\beta_{e_4}=(0,1)$,
    \item[(B)] $\beta_{e_1}+\beta_{e_2}+\beta_{e_5}+\beta_{e_4}=(1,1)$,
    \item[(C)] $\beta_{e_1}+\beta_{e_2}+\beta_{e_5}+\beta_{e_3}=(-1,0)$,
    \item[(D)] $\beta_{e_1}+\beta_{e_2}+\beta_{e_5}=(0,0)$,
    \item[(E)] $\beta_{e_1}+\beta_{e_2}+\beta_{e_5}+\beta_{e_4}+\beta_{e_6}=(1,0)$,
    \item[(F)] $\beta_{e_1}+\beta_{e_2}+\beta_{e_5}+\beta_{e_3}+\beta_{e_6}=(-1,-1)$,
    \item[(G)] $\beta_{e_1}+\beta_{e_2}+\beta_{e_5}+\beta_{e_6}=(0,-1)$.
\end{itemize}

    \begin{figure}[ht]
    \begin{center}
    \begin{tikzpicture}
      \node at (4.6,-1.4) {(A)} ;
      \node at (4.6,0)
      {\scalebox{0.7}{
        \begin{tikzpicture}[line width=0.05cm]
         \coordinate (a1) at (0, 0); 
      \coordinate (a2) at (2, 0); 
      \coordinate (a3) at (-1, 2); 
      \coordinate (a4) at (1, 2);
      \coordinate (a5) at (3,2);

      \draw[
  postaction={
    decorate,
    decoration={
      markings,
      mark=at position 0.5 with {\arrow{>}}
    }
  }
] (a1) -- (a3);
\draw[
  postaction={
    decorate,
    decoration={
      markings,
      mark=at position 0.65 with {\arrow{>}}
    }
  }
] (a1) -- (a4);
\draw[
  postaction={
    decorate,
    decoration={
      markings,
      mark=at position 0.65 with {\arrow{>}}
    }
  }
] (a2) -- (a3);
\draw[
  postaction={
    decorate,
    decoration={
      markings,
      mark=at position 0.65 with {\arrow{>}}
    }
  }
] (a2) -- (a4);
\draw[
  postaction={
    decorate,
    decoration={
      markings,
      mark=at position 0.65 with {\arrow{>}}
    }
  }
] (a1) -- (a5);
\draw[
  postaction={
    decorate,
    decoration={
      markings,
      mark=at position 0.5 with {\arrow{<}}
    }
  }
] (a2) -- (a5);
      \draw [line width=0.05cm, fill=white] (a1) circle [radius=0.15] node [left=5pt] {\large $1$};
      \draw [line width=0.05cm, fill=white] (a2) circle [radius=0.15] node [right=5pt] {\large $2$};
      \draw [line width=0.05cm, fill=white] (a3) circle [radius=0.15] node [left=5pt] {\large $3$};
      \draw [line width=0.05cm, fill=white] (a4) circle [radius=0.15] node [right=5pt] {\large $4$};
      \draw [line width=0.05cm, fill=white] (a5) circle [radius=0.15] node [right=5pt] {\large $5$};
        \end{tikzpicture}
      }};
      \node at (9.2,-1.4) {(B)} ;
      \node at (9.2,0)
      {\scalebox{0.7}{
        \begin{tikzpicture}[line width=0.05cm]
         \coordinate (a1) at (0, 0); 
      \coordinate (a2) at (2, 0); 
      \coordinate (a3) at (-1, 2); 
      \coordinate (a4) at (1, 2);
      \coordinate (a5) at (3,2);

      \draw[
  postaction={
    decorate,
    decoration={
      markings,
      mark=at position 0.5 with {\arrow{>}}
    }
  }
] (a1) -- (a3);
\draw[
  postaction={
    decorate,
    decoration={
      markings,
      mark=at position 0.65 with {\arrow{>}}
    }
  }
] (a1) -- (a4);
\draw[
  postaction={
    decorate,
    decoration={
      markings,
      mark=at position 0.65 with {\arrow{<}}
    }
  }
] (a2) -- (a3);
\draw[
  postaction={
    decorate,
    decoration={
      markings,
      mark=at position 0.65 with {\arrow{>}}
    }
  }
] (a2) -- (a4);
\draw[
  postaction={
    decorate,
    decoration={
      markings,
      mark=at position 0.65 with {\arrow{>}}
    }
  }
] (a1) -- (a5);
\draw[
  postaction={
    decorate,
    decoration={
      markings,
      mark=at position 0.5 with {\arrow{<}}
    }
  }
] (a2) -- (a5);
      \draw [line width=0.05cm, fill=white] (a1) circle [radius=0.15] node [left=5pt] {\large $1$};
      \draw [line width=0.05cm, fill=white] (a2) circle [radius=0.15] node [right=5pt] {\large $2$};
      \draw [line width=0.05cm, fill=white] (a3) circle [radius=0.15] node [left=5pt] {\large $3$};
      \draw [line width=0.05cm, fill=white] (a4) circle [radius=0.15] node [right=5pt] {\large $4$};
      \draw [line width=0.05cm, fill=white] (a5) circle [radius=0.15] node [right=5pt] {\large $5$};
        \end{tikzpicture}
      }};

\node at (0,-4.4) {(C)} ;
        \node at (0,-3)    
      {\scalebox{0.7}{
        \begin{tikzpicture}[line width=0.05cm]
         \coordinate (a1) at (0, 0); 
      \coordinate (a2) at (2, 0); 
      \coordinate (a3) at (-1, 2); 
      \coordinate (a4) at (1, 2);
      \coordinate (a5) at (3,2);

      \draw[
  postaction={
    decorate,
    decoration={
      markings,
      mark=at position 0.5 with {\arrow{>}}
    }
  }
] (a1) -- (a3);
\draw[
  postaction={
    decorate,
    decoration={
      markings,
      mark=at position 0.65 with {\arrow{>}}
    }
  }
] (a1) -- (a4);
\draw[
  postaction={
    decorate,
    decoration={
      markings,
      mark=at position 0.65 with {\arrow{>}}
    }
  }
] (a2) -- (a3);
\draw[
  postaction={
    decorate,
    decoration={
      markings,
      mark=at position 0.65 with {\arrow{<}}
    }
  }
] (a2) -- (a4);
\draw[
  postaction={
    decorate,
    decoration={
      markings,
      mark=at position 0.65 with {\arrow{>}}
    }
  }
] (a1) -- (a5);
\draw[
  postaction={
    decorate,
    decoration={
      markings,
      mark=at position 0.5 with {\arrow{<}}
    }
  }
] (a2) -- (a5);
      \draw [line width=0.05cm, fill=white] (a1) circle [radius=0.15] node [left=5pt] {\large $1$};
      \draw [line width=0.05cm, fill=white] (a2) circle [radius=0.15] node [right=5pt] {\large $2$};
      \draw [line width=0.05cm, fill=white] (a3) circle [radius=0.15] node [left=5pt] {\large $3$};
      \draw [line width=0.05cm, fill=white] (a4) circle [radius=0.15] node [right=5pt] {\large $4$};
      \draw [line width=0.05cm, fill=white] (a5) circle [radius=0.15] node [right=5pt] {\large $5$};
        \end{tikzpicture}
      }};
      \node at (4.6,-4.4) {(D)} ;
      \node at (4.6,-3)
      {\scalebox{0.7}{
        \begin{tikzpicture}[line width=0.05cm]
         \coordinate (a1) at (0, 0); 
      \coordinate (a2) at (2, 0); 
      \coordinate (a3) at (-1, 2); 
      \coordinate (a4) at (1, 2);
      \coordinate (a5) at (3,2);

      \draw[
  postaction={
    decorate,
    decoration={
      markings,
      mark=at position 0.5 with {\arrow{>}}
    }
  }
] (a1) -- (a3);
\draw[
  postaction={
    decorate,
    decoration={
      markings,
      mark=at position 0.65 with {\arrow{>}}
    }
  }
] (a1) -- (a4);
\draw[
  postaction={
    decorate,
    decoration={
      markings,
      mark=at position 0.65 with {\arrow{<}}
    }
  }
] (a2) -- (a3);
\draw[
  postaction={
    decorate,
    decoration={
      markings,
      mark=at position 0.65 with {\arrow{<}}
    }
  }
] (a2) -- (a4);
\draw[
  postaction={
    decorate,
    decoration={
      markings,
      mark=at position 0.65 with {\arrow{>}}
    }
  }
] (a1) -- (a5);
\draw[
  postaction={
    decorate,
    decoration={
      markings,
      mark=at position 0.5 with {\arrow{<}}
    }
  }
] (a2) -- (a5);
      \draw [line width=0.05cm, fill=white] (a1) circle [radius=0.15] node [left=5pt] {\large $1$};
      \draw [line width=0.05cm, fill=white] (a2) circle [radius=0.15] node [right=5pt] {\large $2$};
      \draw [line width=0.05cm, fill=white] (a3) circle [radius=0.15] node [left=5pt] {\large $3$};
      \draw [line width=0.05cm, fill=white] (a4) circle [radius=0.15] node [right=5pt] {\large $4$};
      \draw [line width=0.05cm, fill=white] (a5) circle [radius=0.15] node [right=5pt] {\large $5$};
        \end{tikzpicture}
      }};
      \node at (9.2,-4.4) {(E)} ;
      \node at (9.2,-3)
      {\scalebox{0.7}{
        \begin{tikzpicture}[line width=0.05cm]
         \coordinate (a1) at (0, 0); 
      \coordinate (a2) at (2, 0); 
      \coordinate (a3) at (-1, 2); 
      \coordinate (a4) at (1, 2);
      \coordinate (a5) at (3,2);

      \draw[
  postaction={
    decorate,
    decoration={
      markings,
      mark=at position 0.5 with {\arrow{>}}
    }
  }
] (a1) -- (a3);
\draw[
  postaction={
    decorate,
    decoration={
      markings,
      mark=at position 0.65 with {\arrow{>}}
    }
  }
] (a1) -- (a4);
\draw[
  postaction={
    decorate,
    decoration={
      markings,
      mark=at position 0.65 with {\arrow{<}}
    }
  }
] (a2) -- (a3);
\draw[
  postaction={
    decorate,
    decoration={
      markings,
      mark=at position 0.65 with {\arrow{>}}
    }
  }
] (a2) -- (a4);
\draw[
  postaction={
    decorate,
    decoration={
      markings,
      mark=at position 0.65 with {\arrow{>}}
    }
  }
] (a1) -- (a5);
\draw[
  postaction={
    decorate,
    decoration={
      markings,
      mark=at position 0.5 with {\arrow{>}}
    }
  }
] (a2) -- (a5);
      \draw [line width=0.05cm, fill=white] (a1) circle [radius=0.15] node [left=5pt] {\large $1$};
      \draw [line width=0.05cm, fill=white] (a2) circle [radius=0.15] node [right=5pt] {\large $2$};
      \draw [line width=0.05cm, fill=white] (a3) circle [radius=0.15] node [left=5pt] {\large $3$};
      \draw [line width=0.05cm, fill=white] (a4) circle [radius=0.15] node [right=5pt] {\large $4$};
      \draw [line width=0.05cm, fill=white] (a5) circle [radius=0.15] node [right=5pt] {\large $5$};
        \end{tikzpicture}
      }};

\node at (0,-7.4) {(F)} ;
        \node at (0,-6)    
      {\scalebox{0.7}{
        \begin{tikzpicture}[line width=0.05cm]
         \coordinate (a1) at (0, 0); 
      \coordinate (a2) at (2, 0); 
      \coordinate (a3) at (-1, 2); 
      \coordinate (a4) at (1, 2);
      \coordinate (a5) at (3,2);

      \draw[
  postaction={
    decorate,
    decoration={
      markings,
      mark=at position 0.5 with {\arrow{>}}
    }
  }
] (a1) -- (a3);
\draw[
  postaction={
    decorate,
    decoration={
      markings,
      mark=at position 0.65 with {\arrow{>}}
    }
  }
] (a1) -- (a4);
\draw[
  postaction={
    decorate,
    decoration={
      markings,
      mark=at position 0.65 with {\arrow{>}}
    }
  }
] (a2) -- (a3);
\draw[
  postaction={
    decorate,
    decoration={
      markings,
      mark=at position 0.65 with {\arrow{<}}
    }
  }
] (a2) -- (a4);
\draw[
  postaction={
    decorate,
    decoration={
      markings,
      mark=at position 0.65 with {\arrow{>}}
    }
  }
] (a1) -- (a5);
\draw[
  postaction={
    decorate,
    decoration={
      markings,
      mark=at position 0.5 with {\arrow{>}}
    }
  }
] (a2) -- (a5);
      \draw [line width=0.05cm, fill=white] (a1) circle [radius=0.15] node [left=5pt] {\large $1$};
      \draw [line width=0.05cm, fill=white] (a2) circle [radius=0.15] node [right=5pt] {\large $2$};
      \draw [line width=0.05cm, fill=white] (a3) circle [radius=0.15] node [left=5pt] {\large $3$};
      \draw [line width=0.05cm, fill=white] (a4) circle [radius=0.15] node [right=5pt] {\large $4$};
      \draw [line width=0.05cm, fill=white] (a5) circle [radius=0.15] node [right=5pt] {\large $5$};
        \end{tikzpicture}
      }};
      \node at (4.6,-7.6) {(G)} ;
      \node at (4.6,-6)
      {\scalebox{0.7}{
        \begin{tikzpicture}[line width=0.05cm]
         \coordinate (a1) at (0, 0); 
      \coordinate (a2) at (2, 0); 
      \coordinate (a3) at (-1, 2); 
      \coordinate (a4) at (1, 2);
      \coordinate (a5) at (3,2);

      \draw[
  postaction={
    decorate,
    decoration={
      markings,
      mark=at position 0.5 with {\arrow{>}}
    }
  }
] (a1) -- (a3);
\draw[
  postaction={
    decorate,
    decoration={
      markings,
      mark=at position 0.65 with {\arrow{>}}
    }
  }
] (a1) -- (a4);
\draw[
  postaction={
    decorate,
    decoration={
      markings,
      mark=at position 0.65 with {\arrow{<}}
    }
  }
] (a2) -- (a3);
\draw[
  postaction={
    decorate,
    decoration={
      markings,
      mark=at position 0.65 with {\arrow{<}}
    }
  }
] (a2) -- (a4);
\draw[
  postaction={
    decorate,
    decoration={
      markings,
      mark=at position 0.65 with {\arrow{<}}
    }
  }
] (a1) -- (a5);
\draw[
  postaction={
    decorate,
    decoration={
      markings,
      mark=at position 0.5 with {\arrow{>}}
    }
  }
] (a2) -- (a5);
      \draw [line width=0.05cm, fill=white] (a1) circle [radius=0.15] node [left=5pt] {\large $1$};
      \draw [line width=0.05cm, fill=white] (a2) circle [radius=0.15] node [right=5pt] {\large $2$};
      \draw [line width=0.05cm, fill=white] (a3) circle [radius=0.15] node [left=5pt] {\large $3$};
      \draw [line width=0.05cm, fill=white] (a4) circle [radius=0.15] node [right=5pt] {\large $4$};
      \draw [line width=0.05cm, fill=white] (a5) circle [radius=0.15] node [right=5pt] {\large $5$};
        \end{tikzpicture}
      }};      
      \end{tikzpicture}
      \caption{acyclic orientation}
\label{Figure:Norientation}
\end{center}
  \end{figure}
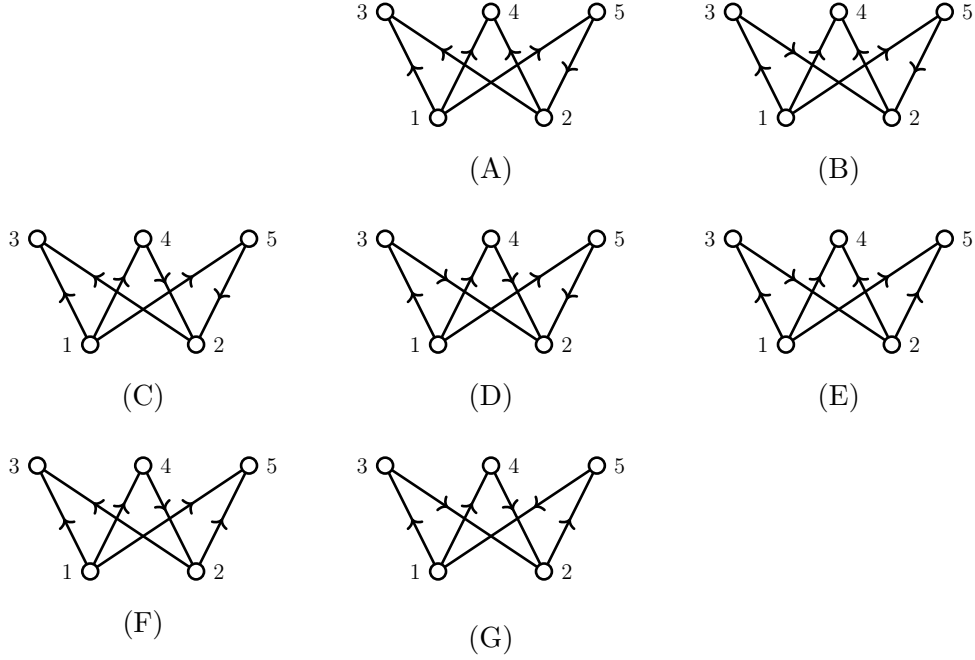
\end{ex}

\subsection*{Acknowledgment}
The author would like to thank Sophie Rehberg for kindly sharing a preprint of her paper \cite{bach2024acyclotopes}.
The author was partially supported by Grant-in-Aid for JSPS Fellows Grant JP25KJ0047.

\bibliographystyle{plain}
\bibliography{References}

\end{document}